\documentclass[12pt]{article}
\pdfoutput=1
\usepackage{amssymb,amsmath,amsthm,latexsym}
\usepackage[cp1251]{inputenc}
\usepackage{dsfont}
\usepackage{amssymb}
\usepackage{indentfirst}
\usepackage{graphicx}
\usepackage{multicol}
\usepackage{amsfonts, array, hhline}
\usepackage{color}
\usepackage{grffile}
\usepackage[export]{adjustbox}

\DeclareMathOperator{\arcsinh}{arcsinh}
\DeclareMathOperator\arctanh{arctanh}
\begin{document}
\title{Ideal polyhedral surfaces in Fuchsian manifolds}
\author{Roman Prosanov \thanks{Supported by SNF grant $200021_-169391$ ``Discrete curvature and rigidity''.}}
\date{}
\AtEndDocument{\bigskip{\footnotesize
\par
  \textsc{Universit\'{e} de Fribourg, Chemin du Mus\'{e}e 23, CH-1700 Fribourg, Switzerland} \par
  \textsc{Moscow Institute Of Physics And Technology, Institutskiy per. 9, 141700, Dolgoprudny, Russia} \par
  \textit{E-mail}: \texttt{rprosanov@mail.ru}
}}
\maketitle

\newtheorem{dfn}{Definition}[section]
\newtheorem{thm}[dfn]{Theorem}
\newtheorem{lm}[dfn]{Lemma}
\newtheorem{crl}[dfn]{Corollary}
\newtheorem{cnj}[dfn]{Conjecture}
\newtheorem{rmk}[dfn]{Remark}
\newtheorem{prop}[dfn]{Proposition}
\renewcommand{\proofname}{Proof}

\renewcommand{\refname}{Bibliography}
\renewcommand{\proofname}{Proof}
\renewcommand{\figurename}{Figure}
\renewcommand{\le}{\leqslant}
\renewcommand{\leq}{\leqslant}
\newcommand{\len}{{\rm len}}
\renewcommand{\ge}{\geqslant}
\renewcommand{\geq}{\geqslant}
\renewcommand{\mathds}{\mathbb}
\newcommand{\e}{\varepsilon}
\newcommand{\R}{\mathbb{R}}
\renewcommand{\H}{\mathbb{H}}
\renewcommand{\L}{\mathbb{L}}
\renewcommand{\S}{\mathbb{S}}
\newcommand{\ds}{{\rm d} \mathbb{S}}
\newcommand{\dist}{{\rm dist}}
\newcommand{\vol}{\textrm{vol}}
\renewcommand{\dist}{{\rm dist}}

\begin{abstract}
Let $S_{g,n}$ be a surface of genus $g > 1$ with $n>0$ punctures equipped with a complete hyperbolic cusp metric. Then it can be uniquely realized as the boundary metric of an ideal Fuchsian polyhedron. In the present paper we give a new variational proof of this result. 
We also give an alternative proof of the existence and uniqueness of a hyperbolic polyhedral metric with prescribed curvature in a given conformal class.
\end{abstract}

\section{Introduction}

\subsection{Theorems of Alexandrov and Rivin}

Consider a convex polytope $P \subset \R^3$. Its boundary is homeomorphic to $S^2$ and carries the intrinsic metric induced from the Euclidean metric on $\R^3$. What are the properties of this metric?

A metric on $S^2$ is called \emph{polyhedral Euclidean} if it is locally isometric to the Euclidean metric on $\R^2$ except finitely many points, which have neighborhoods isometric to an open subset of a cone (an exceptional point is mapped to the apex of this cone). If the conical angle of every exceptional point is less than $2\pi$, then this metric is called \emph{convex}. It is clear that the induced metric on the boundary of a convex polytope is convex polyhedral Euclidean. One can ask a natural question: is this description complete, in the sense that every convex polyhedral Euclidean metric can be realized as the induced metric of a polytope? This question was answered positively by Alexandrov in 1942, see~\cite{Alex}, \cite{AlexB}:

\begin{thm}
\label{alex}
For every convex polyhedral Euclidean metric $d$ on $S^2$ there is a convex polytope $P \subset \R^3$ such that $(S^2, d)$ is isometric to the boundary of $P$. Moreover, such $P$ is unique up to an isometry of $\R^3$.
\end{thm}

Note that $P$ can degenerate to a polygon. In this case $P$ is doubly covered by the sphere. 

The uniqueness part follows from the modified version of Cauchy's global rigidity of convex polytopes. The original proof by Alexandrov of the existence part is not constructive. It is based on some topological properties of the map from the space of convex polytopes to the space of convex polyhedral Euclidean metrics. Another proof was done by Volkov in~\cite{Vo}, a student of Alexandrov, by considering a discrete version of the total scalar curvature.

A new proof of Theorem \ref{alex} was proposed by Bobenko and Izmestiev in \cite{BI}. For a fixed metric they considered a space of polytopes with conical singularities in the interior realizing this metric at the boundary. In order to remove singularities they constructed a functional over this space and investigated its behavior. Such a proof can be turned into a practical algorithm of finding a polytopal realization of a given metric. It was implemented by Stefan Sechelmann. One should note that this algorithm is approximate as it uses numerical methods of solving variational problems, but it works sufficiently well for practical needs.

We turn our attention to hyperbolic metrics on surfaces. By $S_{g,n}$ we mean the surface $S_g$ of genus $g$ with $n$ marked points. Let $d$ be a complete hyperbolic metric of a finite volume with $n$ cusps at the marked points (in what follows we will call it a \emph{cusp metric}). In \cite{Ri} Rivin proved a version of Theorem \ref{alex} for cusp metrics on the 2-sphere:

\begin{thm}
\label{rivin}
For every cusp metric $d$ on $S_{0, n}$ there exists a convex ideal polyhedron $P \subset \H^3$ such that $(S_{0,n}, d)$ is isometric to the boundary of $P$. Moreover, such $P$ is unique up to an isometry of $\H^3$.
\end{thm}

Rivin gave a proof in the spirit of Alexandrov's original proof. Very recently, in~\cite{Spr} Springborn gave a variational proof of Theorem \ref{rivin}.

\subsection{Ideal Fuchsian polyhedra and Alexandrov-type results}

It is of interest to generalize these results to surfaces of higher genus. We restrict ourselves to the case $g>1$ and to cusp metrics. Define ${G:=\pi_1(S_g)}$. Let $\rho: G \rightarrow {\rm Iso^+}(\H^3)$ be a \emph{Fuchsian representation}: an injective homomorphism such that its image is discrete and there is a geodesic plane invariant under $\rho (G)$. Then $F := \H^3 / \rho(G)$ is a complete hyperbolic manifold homeomorphic to $S_g \times \R$. The image of the invariant plane is the so-called convex core of $F$ and is homeomorphic to $S_g$. The manifold $F$ is symmetric with respect to its convex core. The boundary at infinity of $F$ consists of two connected components.

A subset of $F$ is called \emph{convex} if it contains every geodesic between any two its points. It is possible to consider convex hulls with respect to this definition. An \emph{ideal Fuchsian polyhedron} $P$ is the closure of the convex hull of a finite point set in a connected component of $\partial_{\infty}F$. It has two boundary components: one is the convex core and the second is isometric to $(S_{g,n}, d)$ for a cusp metric $d$. We will always refer to the first component as to the \emph{lower boundary} of $P$ and to the second as to the \emph{upper boundary}. The following result can be considered as a generalization of the Alexandrov theorem to surfaces of higher genus with cusp metrics:

\begin{thm}
\label{fil}
For every cusp metric $d$ on $S_{g, n}$, $g>1, n>0$, there exists a Fuchsian manifold $F$ and an ideal Fuchsian polyhedron $P\subset F$ such that $(S_{g,n}, d)$ is isometric to the upper boundary of $P$. Moreover, $F$ and $P$ are unique up to isometry.
\end{thm}

This theorem was first proved by Schlenker in his unpublished manuscript \cite{Sch1}. Another proof was given by Fillastre in \cite{Fil1}. Both these proofs were non-constructive following the original approach of Alexandrov. One of the purposes of the present paper is to give a variational proof of Theorem \ref{fil} by turning it to a finite dimensional convex optimization problem in the spirit of papers \cite{BI}, \cite{FI} and \cite{Spr}. In contrast with the previous proofs, it can be transformed to a numerical algorithm of finding the realization of a given cusp metric as a Fuchsian polyhedron. 

Several authors studied Alexandrov-type questions for hyperbolic surfaces of genus $g>1$ in more general sense. They are collected in the following result of Fillastre \cite{Fil1}. Let $S_{g,n,m}$ be $S_g$ with $n$ marked points and $m$ disjoint closed discs removed. Consider a complete hyperbolic metric $d$ on $S_{g, n, m}$ with cusps and conical singularities at marked points and complete ends of infinite area at removed disks, i.e. boundary components at infinity. One can see an example in the projective model as the intersection of $\overline \H^3$ with a cone having the apex outside of $\overline \H^3$ (\emph{hyperideal point}). Then $(S_{g,n,m},d)$ can be uniquely realized as the induced metric at the upper boundary of a generalized Fuchsian polyhedron (with ideal, hyperideal and conical vertices). It would be interesting to extend the variational technique to this generalization. 
This requires a substantial additional work.

The case of general metrics and $g=0$ was proved by Schlenker in \cite{Sch2}. 
The case of $g>1$ with only conical singularities was first proved in an earlier paper of Fillastre \cite{Fil2} and with only cusps and infinite ends in the paper \cite{Sch3} by Schlenker. The torus case with only conical singularities was the subject of the paper \cite{FI} by Fillastre and Izmestiev. The last paper also followed the scheme of variational proof. All other mentioned works were done in the framework of the original Alexandrov approach. Recently another realization result of metrics on surfaces with conical singularities in a Lorentzian space was obtained by Brunswic in~\cite{Bru}. Although, he worked in a different setting, his methods were close to ours: he also used the discrete Hilbert--Einstein functional and Epstein--Penner decompositions.



\subsection{Discrete conformality}

There is a connection between convex realizations of polyhedral metrics and discrete uniformization problems (see~\cite{BPS} for a detailed exposition). 

Denote the set of marked points of $S_{g,n}$ by $\mathcal B$. Similarly to Euclidean case, we say that a metric $d$ on $S_{g,n}$ is \emph{hyperbolic polyhedral} if it is locally hyperbolic except points of $\mathcal B$ where $d$ can be locally isometric to a hyperbolic cone. Thus, the set of conical singularities is a subset of $\mathcal B$. For $B_i \in \mathcal B$ we define the \emph{curvature} $\kappa_d(B_i)$ to be $2\pi$ minus the cone angle of $B_i$. Note that if $T$ is a geodesic trianglulation of $(S_{g,n}, d)$ with vertices at $\mathcal B$, then $d$ is determined by the side lengths of $T$. 

We say that $T$ is \emph{Delaunay} if when we develop any two adjacent triangles to $\H^2$, the circumbscribed disc of one triangle does not contain the fourth vertex in the interior. We call two polyhedral hyperbolic metrics $d'$ and $d''$ \emph{discretely conformally equivalent} if there exists a sequence of pairs $\{(d_t, T_t)\}_{t=1}^m$, where $d_t$ is a polyhedral hyperbolic metric on $S_{g,n}$, $T_t$ is a Delaunay triangulation of $(S_{g,n}, d_t)$, $d_1=d'$, $d_m=d''$ and for every $t$ either 

(i) $d_t=d_{t+1}$ in the sense that $(S_{g,n}, d_t)$ is isometric to $(S_{g,n}, d_{t+1})$ by an isometry isotopic to identity with respect to $\mathcal B$, or 

(ii) $T_t=T_{t+1}$ and there exists a function $u: \mathcal B \rightarrow \R$ such that for every edge $e$ of $T_t$ with vertices $B_i$ and $B_j$ we have $$\sinh\left(\frac{\len_{d_t}(e)}{2}\right)=\exp(u(B_i)+u(B_j))\sinh\left(\frac{\len_{d_{t+1}}(e)}{2}\right),$$ where $\len_{d_t}(e)$ is the length of $e$ in $d_t$. The following theorem is proved in~\cite{Lu}:

\begin{thm}
\label{Luo}
Let $d$ be a polyhedral hyperbolic metric on $S_{g,n}$ and $\kappa': {\mathcal B\rightarrow (-\infty; 2\pi)}$ be a function satisfying 
\begin{equation}
\label{dGB}
\sum_{B_i \in \mathcal B}\kappa'(B_i)>2\pi(2-2 g).
\end{equation}
Then there exists a unique metric $d'$ discretely conformally equivalent to $d$ such that $\kappa_{d'}(B_i)=\kappa'(B_i)$ for all $B_i\in\mathcal B$.
\end{thm}

The condition~(\ref{dGB}) is necessary by the discrete Gauss-Bonnet theorem: $$\sum_{B_i \in \mathcal B}\kappa'(B_i)=2\pi(2-2 g)+{\rm area}(S_{g,n},d').$$

\begin{crl}
\label{du}
Every polyhedral hyperbolic metric on $S_{g,n}$ is discretely conformally equivalent to a unique hyperbolic metric.
\end{crl}

The existence of $d'$ is proved in~\cite{Lu} in indirect way close to the Alexandrov method. After that it is noted that $d'$ can be found as the critical point of an appropriate strictly convex functional, although the authors do not provide an explicit formula of it. The authors of~\cite{Lu} also observe that Corollary~\ref{du} (\emph{discrete uniformization}) in fact is equivalent to Theorem~\ref{fil}. In Section~\ref{PC} we reformulate Theorem~\ref{Luo} in terms of Fuchsian polyhedra with singularities. We establish the existence and uniqueness in a different way using explicit variational approach combined with geometric observations. 


\subsection{Related work and perspectives}

In \cite{Lei} Leibon gave a characterization of convex ideal Fuchsian polyhedra in terms of their exterior dihedral angles. More precisely, consider $S_{g,n}$ and a triangulation $T$ with vertices at marked points. Assign a real number $\theta_e$ to each edge $e$ of $T$. We call the assignment \emph{Delaunay} if all $\theta_e \in (0; \pi)$. We call it \emph{non-singular} if the sum of $\theta_e$ around a vertex is equal to $2\pi$. Finally, we call it \emph{feasible} if for every subset $X$ of triangles of $T$ we have $\sum\limits_{e\in X}(\pi-\theta_e)>\pi|X|.$ Then the main result of \cite{Lei} can be reformulated as follows:

\begin{thm}
There exists an ideal Fuchsian polyhedron with the face triangulation $T$ and exterior dihedral angles of the upper edges equal to $\theta_e$ if and only if the assignment $\theta_e$ is Delaunay, non-singular and feasible.
\end{thm}

This is similar to the characterization of the dihedral angles of convex ideal polyhedra in the hyperbolic 3-space given by Rivin in~\cite{Ri2}. However, the methods of \cite{Lei} are different from~\cite{Ri2} (although they develop the ideas of another paper of Rivin~\cite{Ri3}). For an assignment $\theta_e$ Leibon defines a conformal class of angle structures on a pair $(S_{g,n}, T)$, which can be parametrized as an open bounded convex polytope in a Euclidean space. He explores the volume functional on this space, which turns out to be strictly concave. Then Leibon shows that critical points of this functional correspond to ideal Fuchsian polyhedra and that under his conditions it attains the maximum in the interior. 

Our approach to Theorem~\ref{fil} can be informally considered as a dual to the mentioned one. Instead of angle structures we consider the space of ideal Fuchsian polyhedra with conical singularities in the interior. In order to remove the singularities we use the so-called \emph{discrete Hilbert-Einstein functional}. There are numerous differences between these two frameworks. For instance, Leibon considers a fixed boundary combinatorics of a polyhedron, but in our case it is allowed to change. There is a hope that it will be possible to use one of these methods in order to provide a new proof of the hyperbolization of 3-manifolds relying  on finite dimensional variational methods only. We refer the reader to the article \cite{FG} considering angle structures in this context and to the survey~\cite{IzmS} discussing perspective applications of the discrete Hilbert-Einstein functional to various geometrization and rigidity problems.

It may be of interest to investigate the following generalization of Theorem~\ref{fil}. Define a \emph{double ideal Fuchsian polyhedron} $P$ as the convex hull of $n>0$ ideal points in one component of $\partial_{\infty}F$ and $m>0$ ideal point in the other one. The boundary of $P$ consists of two components isometric to $(S_{g,n}, d_1)$ and $(S_{g,m}, d_2)$ for two cusp metrics $d_1$ and $d_2$. One can ask if for any two cusp metrics metrics there is a double ideal Fuchsian polyhedron realizing both metrics at its boundary? The answer to this naive question is no. A double ideal Fuchsian polyhedron can be cut into two ideal Fuchsian polyhedra, which have the same metric at their lower boundary. But Theorem~\ref{fil} implies that $(S_{g,n}, d)$ uniquely determines it. Take two cusp metrics such that the corresponding metrics on the lower boundaries are not isometric, then these cusp metrics can not be simultaneously realized by a double ideal Fuchsian polyhedron (and clearly Theorem~\ref{fil} implies that otherwise they can). However, we may consider polyhedra in so-called \emph{quasifuchsian manifolds}.

A representation $\rho$ of $G:=\pi_1(S_g)$ in ${\rm Iso}^+(\H^3)$ is called \emph{quasifuchsian} if it is discrete, faithful and the limit set at the boundary at infinity is a Jordan curve. A hyperbolic manifold $F$ is \emph{quasifuchsian} if it is isometric to $\H^3/\rho(G)$. As in the Fuchsian case, $F$ is homeomorphic to $S_g \times \R$ and has the well-defined boundary at infinity. The \emph{convex core} of $F$ is the image of the convex hull of the limit set under the projection of $\H^3$ onto $F$. It is 3-dimensional if $F$ is not Fuchsian. An \emph{ideal quasifuchsian polyhedron} is the convex hull of $n>0$ ideal points in one component of $\partial_{\infty}F$ and $m>0$ ideal point in the other one (as the convex core is 3-dimensional, the case of vertices belonging to only one boundary component has no significance in contrast to the Fuchsian case).

To state an analog of the uniqueness part of Theorem~\ref{fil} we need a way to connect Teichm\"uller spaces for surfaces with different number of punctures. A \emph{marked cusp metric} is a cusp metric on $S_{g,n}$ together with a \emph{marking} monomorphism $\pi_1(S_g) \hookrightarrow \pi_1(S_{g,n})$. A quasifuchsian manifold $F$ has a canonical identification $\pi_1(F) \simeq \pi_1(S_g)$. For every quasifuchsian polyhedron $P$ it induces monomorphisms $\iota_+$ and $\iota_-$ of $\pi_1(S_g)$ to the fundamental groups of the upper and lower boundary components of $P$ respectively.

\begin{cnj}
\label{cnj1}
Let $d_1$ and $d_2$ be two marked cusp metrics on $S_{g,n}$ and $S_{g,m}$ respectively, $n,m>0$. Then there is a unique ideal quasifuchsian polyhedron $P$ such that one component of its boundary is isometric to $(S_{g,n}, d_1)$, the other one is isometric to $(S_{g,m}, d_2)$ and the compositions of marking monomorphisms with the maps induced by these isometries coincide with $\iota_+$ and $\iota_-$.
\end{cnj}

A similar problem for metrics with conical singularities without uniqueness part was done in the PhD thesis of Slutskiy~\cite{Sl1} by the method of smooth approximation. Later he extended this result to the more general case of metrics of curvature $K \geq -1$ in Alexandrov sense~\cite{Sl2}.

It is interesting to adapt our proof of Theorem~\ref{fil} to this conjecture. As a next step, manifolds with more complicated topology can be considered. 
It is a perspective direction of further research.

\subsection{Overview of the paper}

In Section~\ref{HG} we overview some hyperbolic geometry that will be used in the rest. In Section~\ref{PC} we define our main objects, which are called~\emph{convex prismatic complexes}. These are ideal Fuchsian polyhedra with conical singularities around inner edges incident to cusps and orthogonal to the lower boundary. We formulate the characterization of convex prismatic complexes in terms of their conical angles. This appears to be equivalent to Theorem~\ref{Luo}. The proof of equivalence is postponed to Subsection~\ref{DC}.

The rest can be divided into four major parts. First, we show that convex prismatic complexes can be parametrized by the ``lengths'' of inner edges. This is done in Subsection~\ref{secinj}. Then we prove that in fact any lengths define a complex. To this purpose we investigate a connection with Epstein-Penner decompositions of decorated cusp surfaces in Subsection~\ref{secsur}. In Subsection~\ref{VA1} we introduce the \emph{discrete Hilbert-Einstein functional} and explain how its (unique) critical point gives a convex prismatic complex with prescribed conical angles. Subsection~\ref{VA2} is devoted to a proof of the existence of the critical point. To this purpose we study the behavior of the functional ``near infinity'' by geometric means.

\textbf{Acknowledgments.} The author would like to thank Ivan Izmestiev for numerous useful discussions and his constant attention to this work and to the referee for plenty of valuable suggestions.

\section{Hyperbolic geometry} \label{HG}

\subsection{Hyperboloid model of hyperbolic space}

In this section we fix some notation and mention results from basic hyperbolic geometry that will be used below.

Let $\R^{1,3}$ be the 4-dimensional real vector space equipped with the scalar product $$\langle \overline x, \overline y \rangle := -x_1y_1+x_2y_2+x_3y_3+x_4y_4.$$ By letters with lines above $\overline x$ we denote points of $\R^{1,3}$. Identify $$\H^3 = \{\overline x \in \R^{1,3}: \langle \overline x,\overline x\rangle = -1; x_1 > 0 \}.$$

By $\overline \H^3$ we denote the union of $\H^3$ with its boundary at infinity. We identify $\R^{1,2}$ with the plane $\{\overline x: x_4=0 \}$ and $\H^2$ with $\H^3 \cap \R^{1,2}$.

Define the three-dimensional de Sitter space

$$\ds^3=  \{\overline x \in \R^{1,3}: \langle \overline x,\overline x\rangle = 1 \}$$

and a half of the cone of light-like vectors $$\L = \{\overline x \in \R^{1,3}: \langle \overline x,\overline x\rangle = 0; x_1 > 0 \}.$$

There is a natural correspondence between ideal points of $\overline \H^3$ and generatrices of $\L$. An horosphere is the intersection of $\H^3$ and an affine plane $L$ with light-like normal vector. Then define its polar dual $\overline l \in \L$ by the equation $$\langle\overline l, \overline x\rangle=-1,$$ for all $x \in L$. Slightly abusing the notation, we will use the same letter both for an horosphere and for the defining plane. 

A hyperbolic plane $M$ in $\H^3$ is the intersection of $\H^3$ with a linear two-dimensional time-like subspace of $\R^{1,3}$ with a space-like unit normal vector $\overline m$. Again, in our notation we will not distinguish these planes in $\R^{1,3}$ from the corresponding planes in $\H^3$. (The same also holds for hyperbolic lines in $\H^2$.) However, for points we do distinguish: if $A \in \H^3$, then its defining vector in the hyperboloid model is denoted by $\overline x_A$. 

We will need the following interpretation of scalar products between vectors of $\R^{1,3}$ in terms of distances in $\H^3$ (see \cite{Rat} or \cite{Thu}):

\begin{lm}
\label{scpr}
%
1. If $\overline x_A \in \H^3$ and $\overline l \in \L$, then $$\langle \overline x_A, \overline l \rangle = -e^{\dist(A, L)}$$ where the distance is signed: it is positive if $\overline x_A$ is outside of the horoball bounded by $L$ and negative otherwise.

2. If $\overline m \in \ds^3$ and $\overline l \in \L$, then $$\langle \overline  m, \overline l \rangle = \pm e^{\dist(M, L)}$$ where the distance between a plane and a horosphere is the length of the common perpendicular taken with the minus sign if the plane intersects the horosphere. The sign of the right hand side depends on at which halfspace with respect to $M$ the center of $L$ lies.

3. If $\overline l_1 \in \L$ and $\overline l_2 \in \L$, then $$\langle \overline  l_1, \overline l_2 \rangle = -2e^{\dist(L_1, L_2)}$$ where the distance between two horospheres is the length of the common perpendicular taken with the minus sign if these horospheres intersect.
\end{lm}

From now on we assume that ideal points under our consideration are always equipped with (fixed) horospheres. Under this agreement, we use the word \emph{distance} between two points even in the cases when one of them or both are ideal. In the latter case, by the distance we mean the signed distance between the corresponding horospheres: we write it with the minus sign if the horospheres intersect. In the former case, the distance means the signed distance from the non-ideal point to the  horosphere at the ideal point. Similarly, we speak about the \emph{length} of a segment even if one or two of its endpoints are ideal.

\begin{lm}
\label{ideal}
Let $ABC$ be an ideal hyperbolic triangle with side lengths $a$, $b$ and $c$ respectively and $\alpha_A$ be the length of the part of the horosphere at $A$ inside the triangle. Then $$\alpha^2_A = e^{a-b-c}.$$
\end{lm}

A proof can be found in \cite{Pen}, Proposition 2.8. For the differential formulas in Section~\ref{VA} we will need a semi-ideal version of this lemma: 

\begin{figure}
\begin{center}
\includegraphics[scale=0.3]{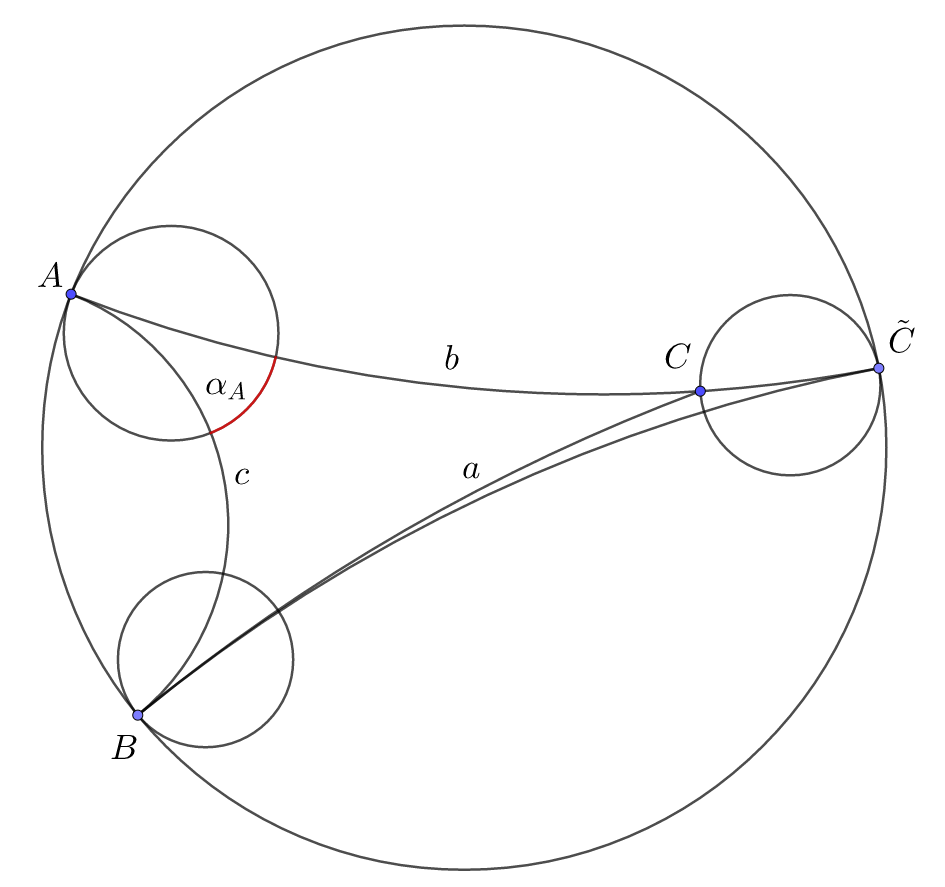}
\caption{To the proof of Lemma~\ref{h3}.}
\label{Pic1}
\end{center}
\end{figure}

\begin{lm}
\label{h3}
Let $ABC$ be a hyperbolic triangle with ideal vertices $A$ and $B$, non-ideal vertex $C$, side lengths $a$, $b$ and $c$ respectively and $\alpha_A$ be the length of the part of the horosphere at $A$ inside the triangle. Then $$\alpha^2_A = e^{a-b-c}-e^{-2b}.$$
\end{lm}

\begin{proof}
Consider the hyperboloid model. Let $\tilde C$ be the intersection of the ray $AC$ with boundary at infinity and put the horocycle at $\tilde C$ passing through $C$ (see Figure~\ref{Pic1}). Denote the side lengths of this new ideal decorated triangle by $\tilde a$, $\tilde b = b$ and $\tilde c = c$. From Lemma~\ref{ideal} it follows that $\alpha^2_A = e^{\tilde a - \tilde b - \tilde c} = e^{\tilde a - b - c}$. Hence, we need to calculate $\tilde a$.

We have $$\overline l_{\tilde C} = \lambda \overline x_C + \mu \overline l _A,$$ $$\langle \overline l_{\tilde C} . \overline l_A \rangle = -\lambda e^b = -2e^b.$$ Hence, we obtain that $\lambda = 2$. Now calculate $$\langle \overline l_{\tilde C} . \overline x_C \rangle = -1 = -\lambda - \mu e^b.$$ We obtain $\mu = -e^{-b}$. We need only to evaluate $$\langle \overline l_{\tilde C} . \overline l_B \rangle = -2e^{\tilde a} = -2 e^a +2e^{c-b}.$$ We get $e^{\tilde a} = e^a - e^{c-b}$. Finally, $\alpha^2_A = e^{a-b-c} - e^{-2b}$.
\end{proof}

\subsection{Epstein--Penner decompositions}
\label{EP}

We recall the concept of Epstein-Penner ideal polygonal decomposition of a decorated cusped hyperbolic surface (see \cite{EP}, \cite{Pen}, \cite{Mar}).

Let $(S_{g,n}, d)$ be a hyperbolic cusp surface. Fix a decoration, i.e. an horocycle at every cusp. Then the space of all decorations can be identified with $\R^n$. A point ${\bf r} \in \R^n$ corresponds to the choice of horocycles at the distances $r_1, \ldots, r_n$ from the fixed ones.

Consider the hyperboloid model of $\H^2$. Represent $(S_{g,n}, d)$ as $\H^2/\Gamma$ where $\H^2 \subset \R^{1,2}$ and $\Gamma$ is a discrete subgroup of ${\rm Iso}^+(\H^2)$ isomorphic to $\pi_1(S_g)$. Take the decoration defined by ${\bf r} \in \R^n$. By $L^1_i, L^2_i, \ldots$ denote horocycles in the orbit of the horocycle at $i$-th cusp under the action of $\Gamma$. By $\mathcal L$ denote the union of their polar vectors $\overline l^k_i$.

Let $C$ be the convex hull of the set $\{\overline l^j_i \}$ in $\R^{1,2}$. Its boundary $\partial C$ is divided into two parts $\partial_l C \sqcup \partial_t C$ consisting of light-like points and time-like points. Below we describe well-known properties of this construction. For proofs we refer to \cite{Mar}, Chapter 5.1.7, \cite{EP} and \cite{Pen}.

\begin{lm}
\label{EPlm}
\begin{itemize}
\item{The convex hull $C$ is 3-dimensional.}
\item{The set $\partial_l C = C \cap \L$ is the set of points $\alpha \overline l^k_i$ for $\alpha \geq 1$.}
\item{Every time-like ray intersects $\partial_t C$ exactly once.}
\item{The boundary $\partial_t C$ is decomposed into countably many Euclidean polygons. The supporting plane containing each polygon is space-like. This decomposition is $\Gamma$-invariant and projects  to a decomposition of $S_{g,n}$ into finitely many ideal polygons.}
\end{itemize}
\end{lm}

\begin{dfn}
This decomposition is called \emph{the Epstein-Penner decomposition} of $(S_{g,n}, d)$ with the decoration $\bf r$.
\end{dfn}

\begin{dfn}
\emph{An Epstein--Penner triangulation} of $(S_{g,n}, d)$ is a geodesic triangulation with vertices at cusps that refines the Epstein--Penner decomposition for some decoration $\bf r$.
\end{dfn}


\subsection{Trapezoids and prisms} \label{PC1}




\begin{dfn}
A \emph {trapezoid} is the convex hull of a segment $A_1A_2 \subset \overline \H^2$ and its orthogonal projection to a line such that the segment $A_1A_2$ does not intersect this line. It is called \emph{ultraparallel} if the line $A_1A_2$ is ultraparallel to the second line. It is called \emph{semi-ideal} if both $A_1$ and $A_2$ are ideal. If some vertices are ideal, then they are equipped with \emph{canonical} horocycles.
\end{dfn}

By $B_i$ denote the image of $A_i$ under the projection, $i=1,2$. We refer to $A_1A_2$ as to the \emph{upper edge}, to $B_1B_2$ as to the \emph{lower edge} and to $A_iB_i$ as to the \emph{lateral edges}. The vertices $A_i$ sometimes are also called \emph{upper} and $B_i$ are called \emph{lower}. We denote by $l_{12}$ the length of $A_1A_2$, by $a_{12}$ the length of $B_1B_2$, by $r_i$ the length of the edge $A_iB_i$, by $\alpha_{12}$ and $\alpha_{21}$ the angles at vertices $A_1$ and $A_2$ (or the lengths of horocycles if the vertices are ideal) and by $\rho_{12}$ the distance from the line $A_1A_2$ to the line $B_1B_2$ in the case of ultraparallel trapezoid. 

\begin{figure}
\begin{center}
\includegraphics[scale=0.2]{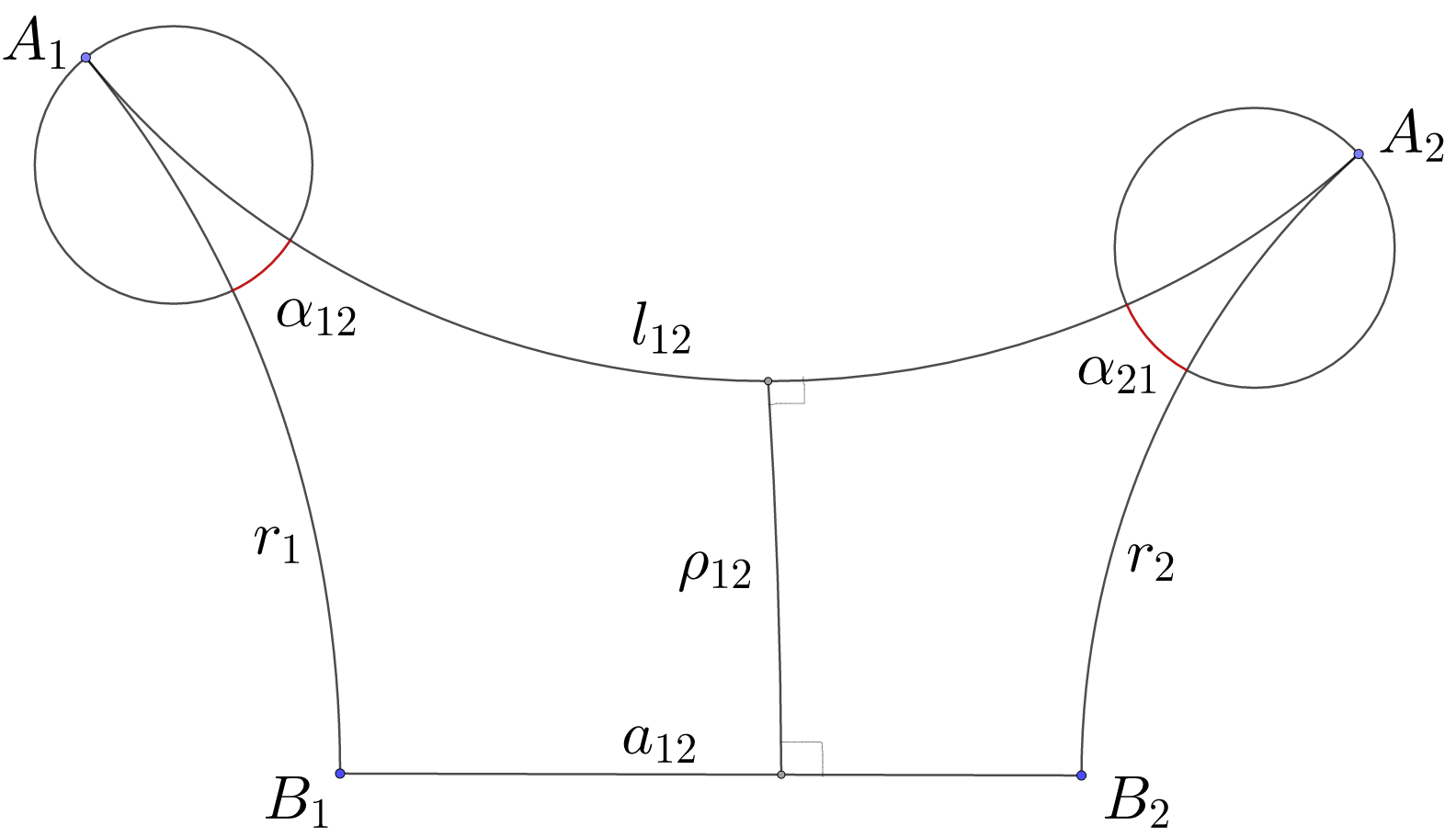}
\caption{A semi-ideal ultraparallel trapezoid. Ideal vertices are equipped with horocycles.}
\label{Pic3}
\end{center}
\end{figure}

\begin{dfn}
A \emph {prism} is the convex hull of a triangle $A_1A_2A_3\subset \overline \H^3$ and its orthogonal projection to a plane such that the triangle $A_1A_2A_3$ does not intersect this plane. It is called \emph{ultraparallel} if the plane $A_1A_2A_3$ is ultraparallel to the second plane. It is called \emph{semi-ideal} if all $A_1$, $A_2$ and $A_3$ are ideal. If some vertices are ideal, then they are equipped with \emph{canonical} horospheres.
\end{dfn}

Similarly to trapezoids, by $B_i$ we denote the image of $A_i$ under the projection, $i=1,2,3$, and we distinguish edges and faces of a prism into \emph{upper, lower} and \emph{lateral}. The lateral faces of a prism are trapezoids. The dihedral angles of edges $B_iB_j$ are equal $\pi/2$. The dihedral angles of edges $A_1A_2$, $A_2A_3$ and $A_3A_1$ are denoted by $\phi_3$, $\phi_1$ and $\phi_2$ respectively. The dihedral angle of an edge $A_iB_i$ is denoted by $\omega_i$. 


\begin{figure}
\begin{center}
\includegraphics[scale=0.3]{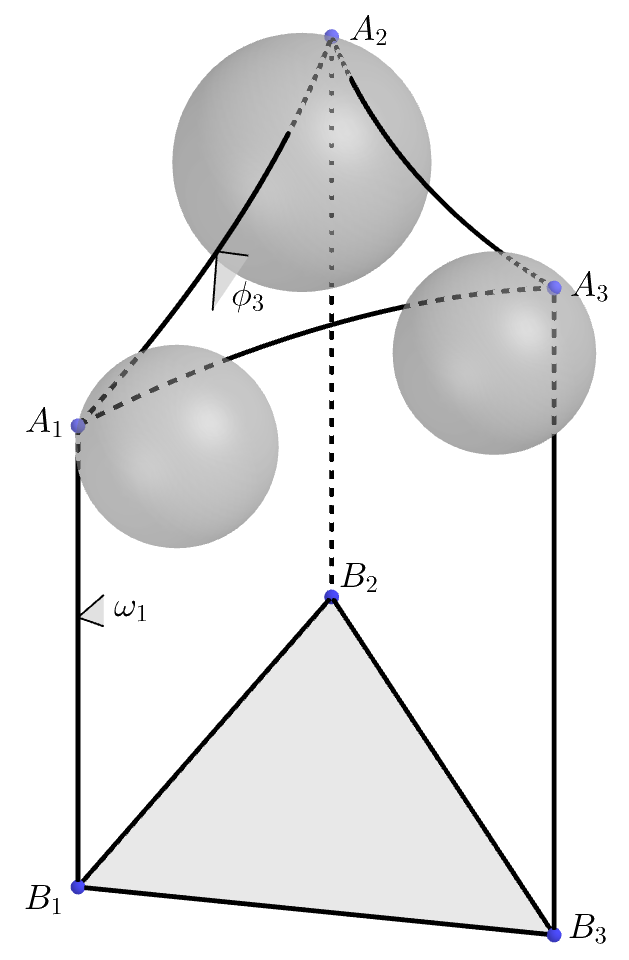}
\caption{A semi-ideal prism. Ideal vertices are equipped with horospheres.}
\label{Pic2}
\end{center}
\end{figure}




In Section~\ref{PC} we will use semi-ideal prisms to construct our main objects: \emph{convex prismatic complexes}. In most cases we need only semi-ideal ultraparallel prisms and trapezoids. The only place, where not ultraparallel prisms appear, is Lemma~\ref{ultrapar} where we prove that actually convex prismatic complexes consist only from ultraparallel ones. The only place, where not semi-ideal prisms are used, is the proof of Lemma~\ref{inj}. In order to prove this lemma, we need to show that ultraparallel prisms (not necessarily semi-ideal) are uniquely determined by the lengths of lateral and upper edges. 

\begin{lm}
\label{dist}
Let $A_1A_2B_2B_1$ be an ultraparallel trapezoid with $A_1, A_2 \in \H^2$ and $\alpha_{21}=\pi/2$. Then $$\sinh(r_1)=\sinh(r_2)\cosh(l_{12}),$$ $$\tanh(a_{12})=\frac{\tanh(l_{12})}{\cosh(r_2)}.$$
\end{lm}

The proof can be found in~\cite{Bus}, Theorem 2.3.1, Formulas (v) and (iv). We need to prove its analogue with one ideal vertex. It will be used further in this subsection to obtain some formulas necessary for Sections~\ref{SP} and~\ref{VA}.

\begin{lm}
\label{distideal}
Let $A_1A_2B_2B_1$ be an ultraparallel trapezoid with $A_1 \in \partial_{\infty} \H^2$, $A_2 \in \H^2$  and $\alpha_{21}=\pi/2$. Then $$e^{r_1}=\sinh(r_2)e^{l_{12}},$$ $$\tanh(a_{12})=\frac{1}{\cosh(r_2)}.$$
\end{lm}

\begin{proof}
Consider the point $A\in A_1B_1$ inside the horodisk at $A_1$. Let $l$ be the length $AA_2$ and $r$ be the length $AB_1$. Then from Lemma~\ref{dist} we have $$\sinh(r_2)=\frac{\sinh(r)}{\cosh(l)}.$$

Now let $r'$ be the modulo of the length $AA_1$, hence $r=r_1+r'$. Extend (if necessarily) $AA_2$ to the intersection point $A'$ with the horocycle at $A_1$, let $l'$ be the modulo of the length $AA'$, $l''$ be the length $A_1A'$ taken with the minus sign if $A_1$ is inside the horodisk, then $l=l''+l'$. Move the point $A$ to $A_1$ and consider the limit of the expression: 
$$\sinh(r_2)=\lim_{A\rightarrow A_1}\frac{\sinh(r_1)\cosh(r')+\cosh(r_1)\sinh(r')}{\cosh(l')\cosh(l'')+\sinh(l')\sinh(l'')}=$$ $$=\lim_{A\rightarrow A_1}\frac{(\sinh(r_1)+\cosh(r_1))e^{r'}}{(\cosh(l'')+\sinh(l''))e^{l'}}=e^{r_1-l_{12}} .$$
This is because $r'-l'$ tends to zero and $l''$ tends to $l_{12}$.

The second formula is obtained similarly.
\end{proof}

The first two corollaries will be used in Section~\ref{PC} to justify the definition of a prismatic complex and in the proof of Lemma~\ref{inj}:

\begin{crl}
In an ultraparallel trapezoid $A_1A_2B_1B_2$ the length of the lower edge is uniquely determined by the lengths of the upper edge and the lateral edges.
\end{crl}

\begin{proof}
Consider $A\in A_1A_2$ that is the closest point to the line $B_1B_2$. Let $B$ be its orthogonal projection to $B_1B_2$. Apply Lemma~\ref{dist} or Lemma~\ref{distideal} to the trapezoids $AA_1B_1B$ and $AA_2B_2B$.
\end{proof}

\begin{crl}
\label{existence}
An ultraparallel trapezoid or an ultraparallel prism is determined up to isometry (mapping canonical horocycles/horospheres, if any, to canonical horocycles/horospheres) by the lengths of the upper edges and the lateral edges. 
\end{crl}

Next one is crucial for Subsection~\ref{VA2}:

\begin{crl}
\label{trapdist}
For a semi-ideal ultraparallel trapezoid we have $$\cosh(a_{12})= 1+\frac{2}{\sinh^2(\rho_{12})}.$$
\end{crl}

\begin{proof}
Let $A\in A_1A_2$ be the closest point to the line $B_1B_2$ and $B$ its projection to this line. Apply the second formula of Lemma~\ref{distideal} to the trapezoid $AA_1B_1B$ and get $\tanh\left(a_{12}/2\right)=\frac{1}{\cosh(\rho_{12})}.$ Then use the formula $$\cosh(a_{12})=\frac{1+\tanh^2\left(a_{12}/2\right)}{1-\tanh^2\left(a_{12}/2\right)}$$ and obtain the desired.
\end{proof}

Using the first formula of Lemma~\ref{distideal} and Corollary~\ref{trapdist} it is straightforward to derive another one, which is necessary for the proof of Lemma~\ref{equiv}:

\begin{crl}
\label{trap} 
For a semi-ideal ultraparallel trapezoid we have
$$\cosh(a_{12})=1+2e^{l_{12}-r_1-r_2}.$$
\end{crl}

From this and Lemma \ref{h3} we also deduce the key fact used in Section~\ref{VA}:

\begin{crl}
\label{alpha}
For a semi-ideal ultraparallel trapezoid we have $$\alpha^2_{12}=e^{r_2-r_1-l_{12}} + e^{-2r_1}.$$
\end{crl}




\section{Prismatic complexes} \label{PC}

Let $(S_{g,n}, d)$ be a hyperbolic cusp-surface with $n$ cusps and $T$ be an ideal geodesic triangulation of $S_{g,n}$ with vertices at cusps. By $E(T)$ and $F(T)$ denote its sets of edges and faces respectively. The set of cusps is denoted by $\mathcal A=\{A_1, \ldots, A_n\}$.  We fix an horodisk at each $A_i$ and until the end of paper we will refer to it as to the \emph{canonical horodisk} at $A_i$ and to its boundary as to the \emph{canonical horocycle}. 

We consider triangulations in a general sense: there may be loops and multiple edges. It is also possible that some triangles have two edges glued together. But without loss of generality, when we consider a particular triangle (or a pair of distinct adjacent triangles), we denote it as $A_iA_jA_h$ (or $A_iA_jA_h$ and $A_jA_hA_g$ respectively).

Suppose that a real weight $r_i$ is assigned to every cusp $A_i$. Denote the weight vector by ${\bf r} \in \R^n$.

\begin{dfn}
A pair $(T, {\bf r})$ is called \emph{admissible} if for every decorated ideal triangle $A_iA_jA_h \in F(T)$ there exists a semi-ideal prism with the lengths of  lateral edges $A_iB_i$, $A_jB_j$, $A_hB_h$ equal to $r_i$, $r_j$ and $r_h$.
\end{dfn}

Let $(T, {\bf r})$ be an admissible pair. For each ideal triangle $A_iA_jA_h \in F(T)$ consider a prism from the last definition. By Corollary~\ref{existence} it is unique up to isometry. Canonical horocycles coming from $(S_{g,n}, d)$ detetmine canonical horospheres at each ideal vertex of the prism.

\begin{dfn}
A \emph{prismatic complex} $K(T, {\bf r})$ is a metric space obtained by glying all these prisms via isometries of lateral faces. We choose glying isometries in such a way that canonical horospheres at ideal vertices of prisms match together.
\end{dfn}

For the sake of brevity, in what follows we will write just \emph{complex} instead of \emph{prismatic complex}. 

If $K(T, {\bf r})$ exists, then it is uniquely determined due to Corollary~\ref{existence}. The solid angle of a prism at an ideal vertex cuts a Euclidean triangle out of the canonical horosphere at this vertex. In a complex $K$ these triangles around an ideal vertex $A_i$ are glued together to form a Euclidean conical polygon, which we call the \emph{canonical horosection} at $A_i$ in $K$.

Every complex is a complete hyperbolic cone manifold with polyhedral boundary. The boundary consists of two components. The union of upper faces forms the {\it upper boundary} coming with a natural isometry to $(S_{g,n}, d)$. This isometry is an important part of the data of $K$. Formally, a complex is not only a metric space, but a pair: metric space plus an isometry of its upper boundary to $(S_{g,n}, d)$. For convenience, in what follows we will just write $(S_{g,n}, d)$ for the upper boundary of $K$. The union of lower faces forms the {\it lower boundary}, which is isometric to $(S_{g, n}, d')$ for a polyhedral hyperbolic metric $d'$ with conical singularities at points $B_i$. We consider $T$ as a geodesic triangulation of both components. The dihedral angle $\tilde{\phi}_e$ of an edge $e \in E(T)$ is the sum of dihedral angles in both prisms containing $e$ and $\tilde \theta_e=\pi-\tilde\phi_e$ is its exterior dihedral angle. (We use tilde in our notation to highlight when we measure angles not in particular prism, but in the whole complex.) The total conical angle $\tilde\omega_i$ of an inner edge $A_iB_i$ is the sum of the corresponding dihedral angles of all prisms containing $A_iB_i$ and $\tilde\kappa_i=2\pi-\tilde\omega_i$ is the curvature of $A_iB_i$. The conical angle of the point $B_i$ in the lower boundary is also equal to $\tilde\omega_i$. 

\begin{dfn}
A complex $K$ is called \emph{convex} if for every upper edge its dihedral angle is at most $\pi$. If $K=K(T, {\bf r})$, then the pair $(T, {\bf r})$ is also called \emph{convex}.
\end{dfn}


Our main aim is to give a variational proof of the following result:

\begin{thm}
\label{angles}
For every cusp metric $d$ on $S_{g, n}$, $g>1, n>0$ with the set of cusps $\mathcal A$ and a function $\kappa': \mathcal A\rightarrow (-\infty; 2\pi)$ satisfying 
$$\sum_{A_i \in \mathcal A}\kappa'(A_i) > 2\pi(2-2g)$$
there exists a unique up to isometry convex complex with the upper boundary isometric to $(S_{g,n},d)$ and the curvature $\tilde\kappa_i$ of each edge $A_iB_i$ equal to $\kappa'(A_i)$.
\end{thm}

\begin{prop}
Theorem~\ref{angles} implies Theorem~\ref{fil}.
\end{prop}

\begin{proof}
By Theorem~\ref{angles}, there exists a complex $K$ with all curvatures $\tilde\kappa_i=0$. We need to show that it is isometric to a Fuchsian polyhedron in a Fuchsian manifold. By Corollary 3.5.3 in~\cite{Mar}, there exists a unique up to isometry complete hyperbolic 3-manifold $F$ without boundary containing $K$ with $\pi_1(F)=\pi_1(K)=\pi_1(S_g)=G$. By Proposition 3.1.3 in~\cite{Mar}, there exists a discrete and faithful representation of $\rho: G\rightarrow {\rm Iso}^+(\H^3)$ such that $F = \H^3/\rho(G).$ The lower boundary of $K$ is lifted to a totally umbilical $\rho$-invariant plane in $\H^3$. Thus, $\rho$ is a Fuchsian representation. It is straightforward that $K$ in $F$ is equal to the closure of the convex hull of the points $A_i$.

Conversely, if $P$ is a Fuchsian polyhedron, then it is isometric to a convex complex $K$ with all curvatures $\kappa_i=0$. Thereby, the uniqueness in Theorem~\ref{angles} implies the uniqueness in Theorem~\ref{fil}.
\end{proof}

In Subsection~\ref{DC} we will show that Theorem~\ref{angles} is equivalent to Theorem~\ref{Luo}. In the rest of this section we prove the following lemma:

\begin{lm}
\label{ultrapar}
Let $K$ be a convex complex. Then each prism of $K$ is ultraparallel.
\end{lm}

\begin{proof}
Embed a prism $A_iA_jA_hB_hB_jB_i \subset K$ in $\H^3$. First, we show that the plane $A_iA_jA_h$ (denote in by $M_1$) can not intersect the plane $B_iB_jB_h$ (denote it by $M_2$) in $\H^3$.

Suppose the contrary. Let these two planes intersect and $l$ be the line of intersection.

The intersection of $M_1$ with $\partial_{\infty}\H^3$ is a circle. The line $l$ divides it into two arcs. All points $A_i$, $A_j$ and $A_h$ belong to the same arc and one of them lies between the two others. Assume that this point is $A_i$. Then we call the edge $A_jA_h$ \emph{heavy} and two other edges \emph{light} (see Figure~\ref{Pic4}).

\begin{figure}
\begin{center}
\includegraphics[scale=0.3]{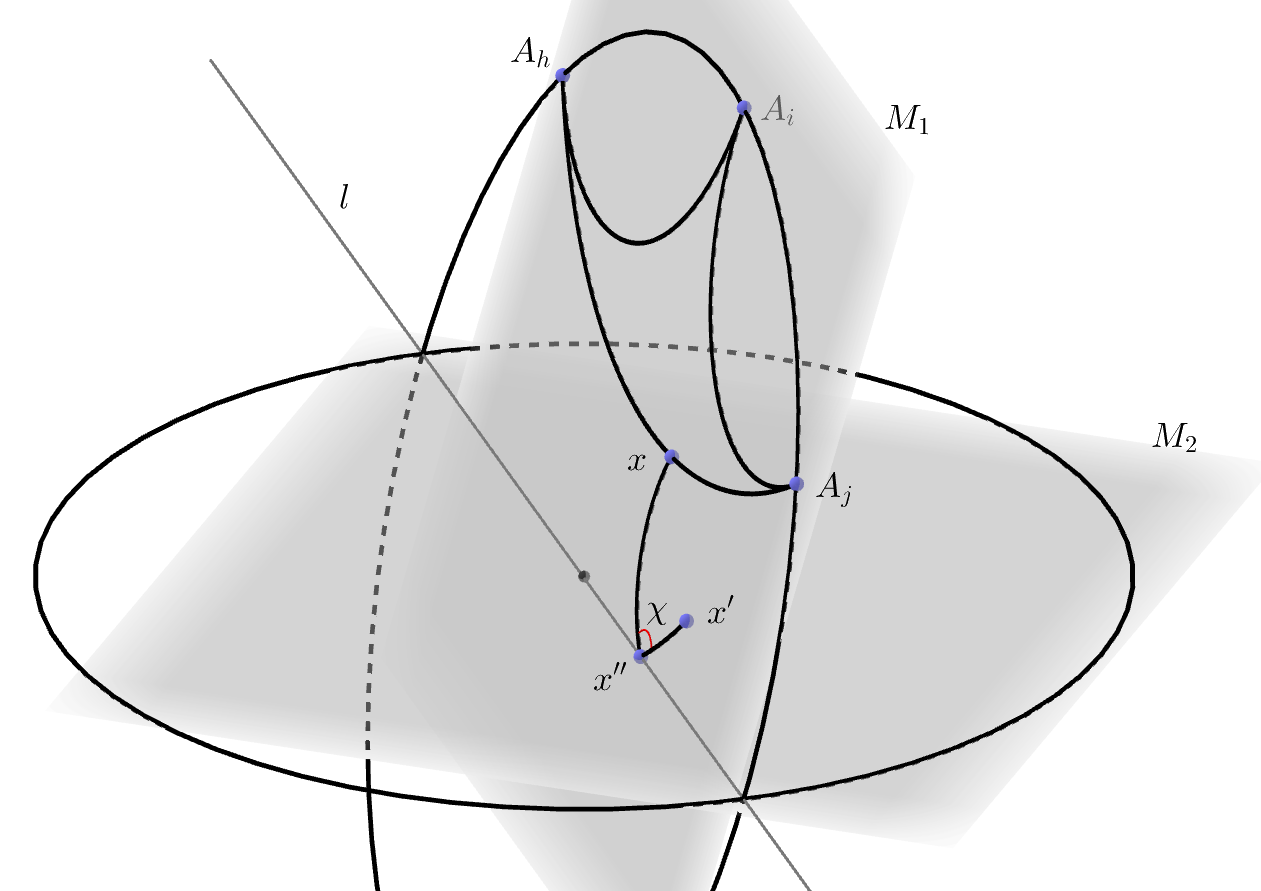}
\caption{To the proof of Lemma~\ref{ultrapar}.}
\label{Pic4}
\end{center}
\end{figure}

Let $\chi$ be the dihedral angle between $M_1$ and $M_2$. For every $x \in M_1$, we have $$\sinh{\rm dist}(x, M_2)=\sinh{\rm dist}(x, l)\sin(\chi),$$ by the law of sines in a right-angled hyperbolic triangle.

It follows that the distances from the light edges to $M_2$ are both strictly bigger than the distance from the heavy edge.
For the dihedral angles of the upper edges we have $\phi_i > \pi/2$ and $\phi_j$,~$\phi_h < \pi/2$.

Indeed, let $x \in A_jA_h$ be the nearest point from this edge to $M_2$, ${x' \in M_2}$ and $x'' \in l$ be the bases of perpendiculars from $x$ to $M_2$ and to $l$. Then $\angle xx'x'' = \pi/2$, $\angle x'xx'' < \pi/2$ and $\phi_i = \pi - \angle x'xx'' > \pi/2$. Next, we consider the ideal vertex $A_j$. Using that the sum of three dihedral angles at one vertex is equal $\pi$ we obtain $$\omega_j+\phi_i+\phi_h = \pi.$$ It implies that $\phi_h < \pi/2$. Similarly, $\phi_j < \pi/2$.

Edge $A_jA_h$ can not be glued in $T$ neither with the edge $A_iA_j$ nor with $A_iA_h$ because these edges have bigger distances to the lower face than $A_jA_h$. Therefore, there is another triangle $A_jA_hA_g \in T$ containing $A_jA_h$. Embed the corresponding prism $A_jA_hA_gB_gB_hB_j \subset K$ in $\H^3$ in such a way that it is glued with the former prism over the face $A_jA_hB_hB_j$ via an orientation-reversing isometry. Then $B_g\ \in M_2$.

The total dihedral angle at $A_jA_h$ is less or equal than $\pi$. 
Hence, the plane $A_jA_hA_g$ also intersects $M_2$. Therefore, the light edges and the heavy edge are defined for the new prism in the same way. Moreover, it is clear that in this prism the edge $A_jA_h$ is light. Hence, we see that the distance from the new heavy edge to $M_2$ is strictly less than the distance from $A_jA_h$. Now for this edge we choose the next prism containing it and continue this process. The distances from the heavy edges to $M_2$ are strictly decreasing in the obtained sequence of prisms. But the number of prisms in $K$ is finite. We get a contradiction.

It remains to consider the case when the upper face is asymptotically parallel to the lower face. It is clear that the contradiction is just the same.
\end{proof}

\begin{rmk}
This is equivalent to the following statement in the language of~\cite{Lu} (Theorem 14): if $T$ is a Delaunay triangulation of a hyperbolic polyhedral metric $d$ on $S_{g,n}$, then each triangle has a compact circumcircle (i.e. not horocyclic or hypercyclic).
\end{rmk}

\section{The space of convex complexes} \label{SP}


Denote by $\mathcal{K}$ the set of all convex complexes with the upper boundary isometric to $(S_{g,n}, d)$ considered up to marked isometry (an isometry between $K_1$ and $K_2$ is called \emph{marked} if it induces an isometry from $(S_{g,n}, d)$ to itself isotopic to identity with respect to $\mathcal A$). In this section we are going to give $\mathcal K$ a nice parametrization.

Every $K \in \mathcal K$ can be represented as $K(T, {\bf r})$. Clearly, if $K' = K(T', {\bf r}')$, $K'' = K(T'', {\bf r}'')$ and ${\bf r}' \neq {\bf r}''$, then complexes $K'$ and $K''$ are not marked isometric. This defines a map, which we denote by ${\bf r}: \mathcal K \rightarrow \R^n$ abusing the notation. In Subsection \ref{SP}.1 we prove

\begin{lm}
\label{inj}
Let $(T', {\bf r})$ and $K''=(T'', {\bf r})$ be two convex pairs. Then the complexes $K'=K(T', {\bf r})$ and $K''=K(T'', {\bf r})$ are marked isometric.
\end{lm}

\begin{crl}
The map ${\bf r}: \mathcal K \rightarrow \R^n$ is injective.
\end{crl}

Hence, $\mathcal K$ can be identified with a subset of $\R^n$. In Subsection \ref{SP}.2 we show that

\begin{lm}
\label{sur}
The image ${\bf r}(\mathcal K) = \R^n$.
\end{lm}


\subsection{The proof of Lemma \ref{inj}}
\label{secinj}

First, we introduce some machinery. Recall that the upper boundary of a convex complex $K=K(T, {\bf r})$ is identified with $(S_{g,n}, d)$. Define a function $$\rho_K: S_{g,n} \rightarrow \R_{>0}$$ to be the distance from $X \in S_{g,n}$ to the lower boundary of $K$.

\begin{dfn}
The function $\rho_K$ is called \emph{the distance function} of $K$.
\end{dfn}



We need an explicit expression for $\rho_K$. Let $s: [x^0; x^1] \rightarrow S_{g,n}$ be a geodesic segment parametrized by length such that its image is contained in a triangle of $T$. Consider a trapezoid obtained from the segment $s$ and its projection to the lower boundary. Develop it to $\H^2$ and extend its upper and lower boundaries to two ultraparallel lines $\psi_1$ and $\psi_2$ respectively. The first formula of Lemma~\ref{dist} shows that $\rho_K\circ s$ has the form 
\begin{equation}
\label{pd}
\arcsinh(b\cosh(x-a))
\end{equation}
for some real number $a$ and positive real number $b$. Indeed, $b$ is the hyperbolic sine of the distance between $\psi_1$ and $\psi_2$ and $(x-a)$ is the distance from a point of $\psi_1$ to the closest point on $\psi_1$ to $\psi_2$.

\begin{dfn}
The function $\rho: \R \rightarrow \R$ of the form (\ref{pd}) is called a \emph{distance-like function}.
\end{dfn}

We establish some basic properties of distance-like functions that we need for the proof of Lemma~\ref{inj}.

\begin{prop}
\label{claim1}
Let $$\rho_1(x)=\arcsinh(b_1\cosh(x-a_1)),$$ $$\rho_2(x)=\arcsinh(b_2\cosh(x-a_2))$$ be two distance-like functions such that the pair $(a_1, b_1)$ is distinct from the pair $(a_2, b_2)$. Then the equation $\rho_1(x)=\rho_2(x)$ has at most one solution.
\end{prop}

\begin{proof}
Note that if $a_1 \neq a_2$, then the function $\frac{\cosh(x-a_1)}{\cosh(x-a_2)}$ has the derivative $\frac{\sinh(a_2-a_1)}{\cosh^2(x-a_2)}$, which has constant nonzero sign. Therefore, the equation $\rho_1(x)=\rho_2(x)$, which is equivalent to $\frac{b_2}{b_1}=\frac{\cosh(x-a_1)}{\cosh(x-a_2)}$, has at most one solution.

If $a_1=a_2$, but $b_1\neq b_2$, then for all $x$, $\rho_1(x)\neq\rho_2(x)$.
\end{proof}

\begin{prop}
\label{claim2} 
Let $$\rho_1(x)=\arcsinh(b_1\cosh(x-a_1)),$$ $$\rho_2(x)=\arcsinh(b_2\cosh(x-a_2))$$ be two distance-like functions such that for $x_0 \in \R$ we have $\rho_1(x_0)=\rho_2(x_0)$ and $\rho_1'(x_0)> \rho_2'(x_0)$. Then $a_2 > a_1$.
\end{prop}

\begin{proof}
We have $$\rho'_1(x_0)=\frac{b_1\sinh(x_0-a_1)}{\sqrt{b_1^2\cosh^2(x_0-a_1)+1}}> \frac{b_2\sinh(x_0-a_2)}{\sqrt{b_2^2\cosh^2(x_0-a_2)+1}}=\rho'_2(x_0).$$ 
Using $\rho_1(x_0)=\rho_2(x_0)$ and the fact that $b_1$, $b_2$ are positive we obtain that this is equivalent to $$\tanh(x_0-a_1)=\frac{\sinh(x_0-a_1)}{\cosh(x_0-a_1)}> \frac{\sinh(x_0-a_2)}{\cosh(x_0-a_1)}=\tanh(x_0-a_2).$$ 
The function $\tanh$ is strictly increasing. This shows the desired statement.
\end{proof}

\begin{prop}
\label{claim3}
Let $\psi_1$ and $\psi_2$ be two distinct geodesic lines in $\H^2$ meeting at a point $A \in \partial_{\infty} \H^2$ and ultraparallel to a line $\psi_0$. Let $A$ be decorated by an horocycle and $\psi_1$, $\psi_2$ be parametrized by the (signed) distance to this horocycle. Denote the distance functions from $\psi_1$ and $\psi_2$ to $\psi_0$ by $$\rho_1(x)=\arcsinh(b_1\cosh(x-a_1)),$$  $$\rho_2(x)={\arcsinh(b_2\cosh(x-a_2))}$$ respectively. Then ${\rho_1(x)-\rho_2(x)}$ has a constant nonzero sign. Besides, if $\rho_1(x)>\rho_2(x)$, then $a_2>a_1$.
\end{prop}

\begin{proof}
The first claim is straightforward. For the second claim, let $C_i \in \psi_i$ be the closest point from $\psi_i$ to $\psi_0$ for $i=1,2$. Recall that $b_i$ is the hyperbolic sine of the distance from $\psi_i$ to $\psi_0$ and $a_i$ is the distance from $C_i$ to the horocycle. Observe that the sign of ${\rho_1(x)-\rho_2(x)}$ is the sign of $b_1-b_2$. Also note that $a_1=a_2$ if and only if $b_1=b_2$ (if and only if lines $\psi_1$ and $\psi_2$ coincide) and $a_i$ decreases as $b_i$ grows.
\end{proof}

Consider $K=K(T, \bf r)$ and a geodesic $s: \R \rightarrow S_{g,n}$ joining two cusps (or, possibly, a cusp with itself), distinct from the edges of $T$ and parametrized by (signed) distance to the horocycle decorating one of its endpoints. Let $$\R =(-\infty; x_0] \cup [x_0; x_1]\cup\ldots\cup[x_k; +\infty)$$ be the subdivision induced by intersections of $s$ with the strictly convex edges of $T$. For convenience, we set $x_{-1}=-\infty$ and $x_{k+1}=+\infty$. For every $i=0,\ldots, k+1$, the restriction of $\rho_K\circ s$ to $(x_{i-1}; x_i)$ has the form~(\ref{pd}): $\arcsinh(b_i\cosh(x-a_i)).$ As our subdivision is induced by intersections with only strictly convex edges of $T$, each pair $(a_i, b_i)$ is distinct from the pair $(a_{i+1}, b_{i+1})$. The intersection points $x_0, \ldots, x_k$ are \emph{kink} points of $\rho_K\circ s$ in the sense that $\rho_K\circ s$ is not differentiable at these points, but both the left derivative and the right derivative exist. It is clear that convexity of $K$ means that at every kink point $x_i$ the left derivative of $\rho_K\circ s$ is strictly greater than the right derivative. We call a function of this type a \emph{piecewise distance-like function}.

\begin{figure}
\begin{center}
\includegraphics[scale=0.15]{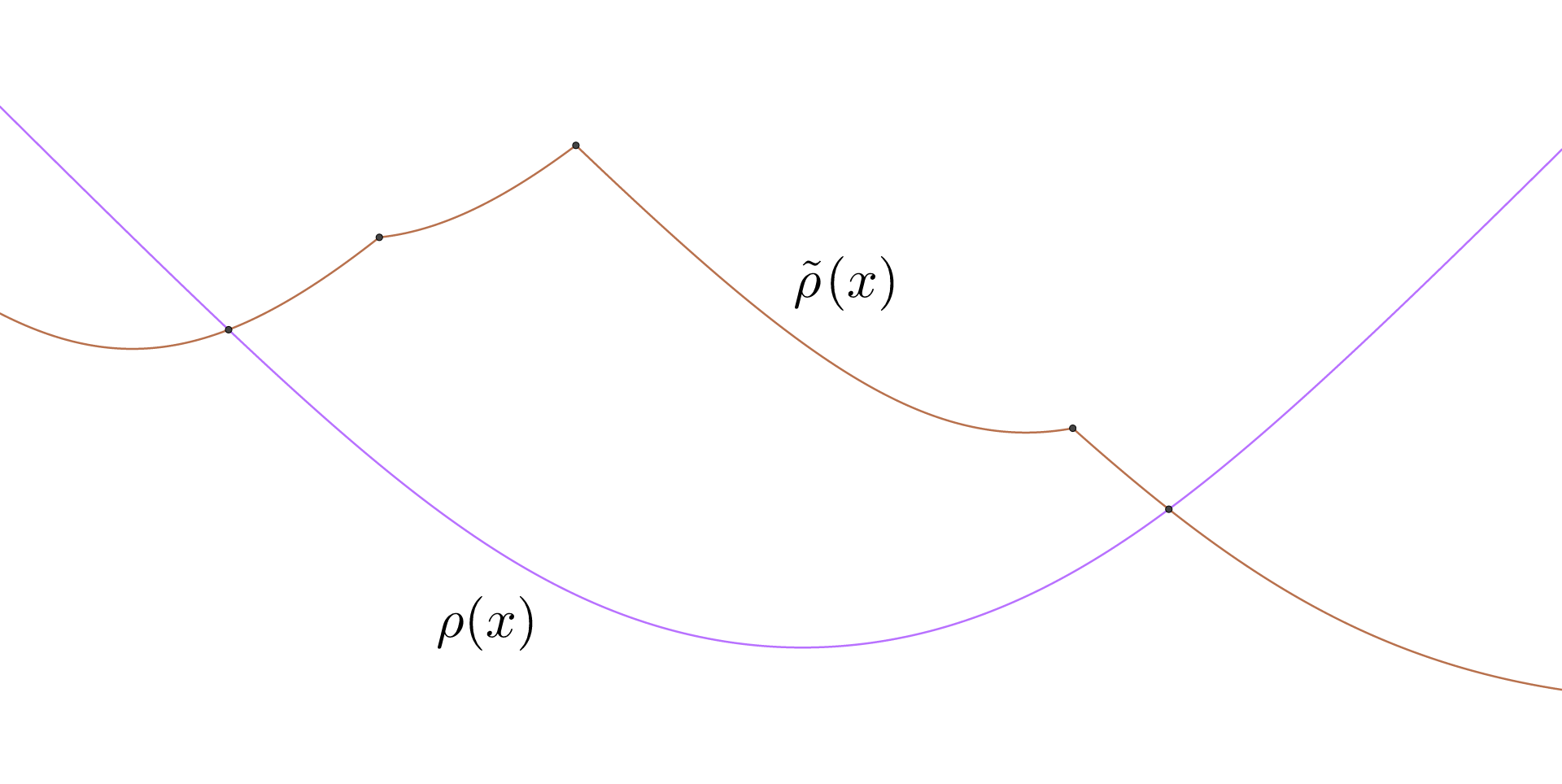}
\caption{Graphics of distance-like and piecewise distance-like functions.}
\label{Pic6}
\end{center}
\end{figure}

\begin{prop}
\label{claim4}
Let $\rho(x)=\arcsinh(b\cosh(x-a))$ be a distance-like function and $\tilde \rho(x)$ be a piecewise distance-like function. By $\rho_i(x)$ we denote the distance-like function, which coincides with $\tilde\rho(x)$ on $(x_{i-1}; x_i)$. Assume that for some $i$ and for all $x\in \R$ we have $\rho(x)>\rho_i(x)$. Then $\rho(x) > \tilde \rho(x)$ for all $x\in \R$.
\end{prop}

\begin{proof}
If $i \neq 0$, then using Proposition~\ref{claim1} for any $x \in [x_{i-2}; x_{i-1})$ we obtain $\rho_i(x) > \rho_{i-1}(x)$. Similarly, if $i \neq k+1$, then for any $x\in (x_i; x_{i+1}]$, $\rho_i(x)>\rho_{i+1}(x)$. By induction, for any $x\notin [x_{i-1}; x_i]$ we get $\rho_i(x)>\tilde \rho(x).$ Then for all $x\in \R$ we have $\rho(x)>\tilde \rho(x).$
\end{proof}

Now we can prove Lemma \ref{inj}:

\begin{proof}

Let $A$ be the intersection point of an edge $e'$ of $T'$ with an edge $e''$ of $T''$. The edge $e'$ is a geodesic in $(S_{g,n}, d)$, we parametrize it by the distance to the horocycle at one of its endpoints and look at the restriction of the distance function $\rho_{K'}$. We denote the resulting distance-like function by $\rho$. Consider also the piecewise distance-like function $\tilde\rho(x)$ obtained from the restriction of the distance function $\rho_{K''}$ to $e'$. We prove that $\tilde\rho(x)\geq \rho(x)$ for every $x \in \R$.

As before, $\R =(-\infty; x_0] \cup [x_0; x_1]\cup\ldots\cup[x_k; +\infty)$ is the decomposition for $\tilde \rho$, and $\rho_i$ is a distance-like function, which coincides with $\tilde \rho$ on $(x_{i-1}; x_i)$. By Proposition~\ref{claim3}, the sign of $\rho(x)-\rho_0(x)$ is constant. Suppose that $\rho_0(x)<\rho(x)$. Then, by Proposition~\ref{claim3}, we have $a<a_0$. By Proposition~\ref{claim4}, we see that for all $x\in\R$, $\tilde\rho(x)<\rho(x)$ and by Proposition~\ref{claim2} and induction we get $a_0<a_1<\ldots< a_{k+1}$. Therefore, $a<a_{k+1}$. On the other hand, consider the parametrization of $e'$ by distance to the horosphere at another endpoint. Then the new distance functions are $\rho(l_{e'}-x)$ and $\tilde \rho(l_{e'}-x)$, where $l_{e'}$ is the length of edge $e'$. We apply Proposition~\ref{claim3} one more time and obtain $l_{e'}-a<l_{e'}- a_{k+1}$. This is equivalent to $a_{k+1}<a$ and gives a contradiction. We also obtain the same contradiction, if we suppose that $\rho_k(x) < \rho(x)$.

Now suppose that for $1\leq i \leq k$ and a point $y \in [x_{i-1}; x_i]$ we have $\tilde\rho(y) < \rho(y)$. By Proposition~\ref{claim1} and induction we know that $\rho_i$ is strictly bigger than $\tilde \rho$ outside of $[x_{i-1}; x_i]$. Also by Proposition~\ref{claim1} we see that either for all $x \in (-\infty; y]$ or for all $x \in [y; +\infty)$ we have $\rho(x)>\rho_i(x)$. Altogether this gives us that either $\rho_0(x)<\rho(x)$ or $\rho_k(x) < \rho(x)$. In any way we reduced ourselves to a previous case.

Thereby, $\tilde\rho(x)\geq \rho(x)$ and we infer that $\rho_{K''}(A) \geq \rho_{K'}(A)$. Similarly, we obtain that $\rho_{K'}(A) \geq \rho_{K''}(A)$ if consider the distance functions restricted to the edge $e''$. Therefore, $\rho_{K'}(A)=\rho_{K''}(A)$.

Let ${\mathcal A}^*$ be the set of all cusps and all intersection points of the edges of $T'$ with the edges of $T''$. The union $E(T') \cup E(T'')$ decomposes $S_{g,n}$ into convex geodesic polygons. We subdivide each polygon into geodesic triangles and obtain a triangulation $T$ with the vertex set ${\mathcal A}^*$ refining both $T'$ and $T''$. This triangulation induces a subdivision of both $K'$ and $K''$ into prisms (not all semi-ideal). Two corresponding prisms are isometric because of Corollary \ref{existence}. In turn we extend these isometries to a marked isometry of $K'$ to $K''$.

\end{proof}

From now on we denote by $K(\bf r)$ the convex complex defined by ${\bf r} \in \R^n$ if it exists. If $K=K(T,{\bf r})$, we call an edge of $T$ \emph{flat} in $K$ if its dihedral angle is equal to $\pi$. Otherwise, we call it \emph{strictly convex} in $K$. Lemma~\ref{inj} implies
\begin{crl}
\label{diffl}
If $(T', {\bf r})$ and $(T'', {\bf r})$ are two convex pairs, then each strictly convex edge of $T'$ in $K({\bf r})$ is an edge of $T''$ and vice versa. Hence, $T'$ and $T''$ differ only in flat edges of $K({\bf r})$.
\end{crl}


\begin{dfn}
Let $K$ be a complex. We say that two point of its upper boundary \emph{lie in the same face} if they can be connected by a path that does not intersect strictly convex edges of $K$. This defines an equivalence relation. A \emph{face} of $K$ is the union of all points in an equivalence class.
\end{dfn}

Thereby, we obtain the decomposition of the upper boundary of $K$ into faces. By definition, a point of a strictly convex edge does not belong to any face. A pair $(T, {\bf r})$ is convex if and only if $T$ refines the face decomposition of $K(\bf r)$.






\begin{lm}
\label{splcn}
A face $\Pi$ of a convex complex $K$ is simply connected.
\end{lm}

\begin{proof}
First, we prove that if $\Pi$ is not simply connected, then there is a closed geodesic in $\Pi$. 

Let $T$ be a triangulation such that $K=K(T, {\bf r})$. Choose a simple homotopically nontrivial closed curve $\psi$ in $\Pi$ that is transversal to interior edges of $\Pi$. Develop all triangles of $T$ that intersect $\psi$ to $\H^2$ (each triangle is developed once). We obtain an ideal polygon $P$. The triangulation $T$ is lifted to a triangulation of $P$. All inner edges of $P$ are lifts of flat edges of $K$.

Let $\tau: P \rightarrow \Pi$ be the projection. It is injective in the interior, but glue at least two boundary edges of $P$ to a flat edge of $\Pi$. Denote them by $AB$ and $CD$: $\tau(AB)=\tau(CD)$, $\tau(A)=\tau(C)$ and $\tau(B)=\tau(D)$ (note that $A$ may coincide with $C$ and $B$ may coincide with $D$). 
For a point $X \in AB$ there is a unique point $Y \in CD$ such that $\tau(Y)=\tau(X)$. A hyperbolic segment $XY$ projects to almost a geodesic loop in $\Pi$. It can have a kink point only at $\tau(X)=\tau(Y)$. Clearly, $\tau(XY)$ is a closed geodesic if and only if $\angle BXY + \angle XYD = \pi$. It is clear that as $X$ tends to $B$, the point $Y$ tends to $D$ and this sum tends to $2\pi$. Similarly, as $X$ tends to $C$, this sum tends to 0. Therefore, there exists $X$ such that this sum is equal to $\pi$. In this case $\tau(XY)$ is a closed geodesic $\psi' \subset \Pi$. 

Consider the distance function $\rho_K$. Its restriction to $\psi'$ must be periodic, because $\psi'$ is a closed geodesic. On the other hand $\psi'$ intersects no strictly convex edges. Therefore, the restriction of $\rho$ to $\psi'$ has the form (\ref{pd}), which is not periodic. We obtain a contradiction.
\end{proof}


\begin{crl}
For every ${\bf r} \in \R^n$ there are finitely many triangulations $T$ such that the pair $(T, {\bf r})$ is convex.
\end{crl}



For a triangulation $T$ denote by $\mathcal K(T) \subset \mathcal K =\R^n$ the set of all ${\bf r} \in \R^n$ such that the pair $(T, {\bf r})$ is convex. This defines a subdivision of $\R^n$ into cells corresponding to different admissible triangulations. 
It is evident that the boundary of $\mathcal K(T)$ is piecewise analytic.

\subsection{Proof of Lemma \ref{sur}}
\label{secsur}

\begin{lm}
The pair $(T, {\bf r})$ is convex if and only if $T$ is an Epstein--Penner triangulation for ${\bf r} \in \R^n$.
\end{lm}

Clearly, this lemma implies Lemma \ref{sur}. Moreover, the face decomposition of $K({\bf r})$ is exactly the Epstein--Penner decomposition of $(S_{g,n}, d)$ with decoration defined by ${\bf r}$ and the subdivision $\mathcal K = \bigcup \mathcal K(T)$ is the Epstein--Penner subdivision of~$\R^n$.

\begin{proof}
Let ${\bf r} \in \R^n$ and let $T$ be one of its Epstein--Penner triangulations. Represent $S_{g,n}$ as $\H^2/\Gamma$ and lift $T$ to a triangulation $\hat T$ of $\H^2$. As in Subsection~\ref{EP}, we denote the polar vector to a horosphere $L$ by $\bar l$, the Epstein-Penner convex hull by $C$ and the set of vertices of $C$ by $\mathcal L$. Let $\Delta = A_iA_jA_h$ be a triangle of $\hat T$, $L_i$, $L_j$ and $L_h$ be the horocycles at $A_i$, $A_j$ and $A_h$ defined by ${\bf r}$, i.e. at distances equal to $r_i$, $r_j$, $r_h$ from the canonical ones. The affine plane $M=M(\Delta) \subset \R^{1,2}$ spanned by the points $\overline l_i$, $\overline l_j$ and $\overline l_h$ is a supporting plane of $C$. By Lemma~\ref{EPlm}, $M$ is space-like, which means that its normal $\overline m$ (in the direction of $C$) is time-like. Let $\mathcal L_M = M \cap \mathcal L$. As $M$ is a supporting plane to $C$, for $\overline l \in \mathcal L$ we have $$\langle \overline m, \overline l \rangle =\begin{cases} -1 &\mbox{if } \overline l \in \mathcal L_M,\\ < -1 &\mbox{otherwise.} \end{cases}$$

Now assume that $\R^{1,2} \hookrightarrow \R^{1,3}$ as $\{ \overline x \in \R^{1,3}: x_4=0 \}$ and $\H^2$ is embedded in $\H^3$ respectively. Extend each horocycle to an horosphere. We continue to denote them by $L_i$. Let $\overline n$ be the intersection point of $\ds^3$ with the ray $\overline m + \lambda \overline e_4$, $\lambda >0$.

By construction we have $$\langle \overline n, \overline l \rangle = \begin{cases} -1 &\mbox{if } \overline l \in \mathcal L_M,\\< -1 &\mbox{otherwise.} \end{cases}$$

Let $N=N(\Delta) \subset \H^3$ be the plane obtained from the time-like linear plane in $\R^{1,3}$ 
orthogonal to $\overline n$. From Lemma \ref{scpr} we see that each horosphere $L$ (with $\overline l \in \mathcal L$) lies in the closed halfspace $N_-$ and $N$ is tangent to $L$ if and only if $\overline l \in \mathcal L_M$ (otherwise $N$ does not intersect $L$). We summarize it in the following description (we proved only in one direction, but the converse is clear):

\begin{prop}
\label{EpP}
A triangle $A_iA_jA_h$ is contained in a face of the Epstein--Penner decomposition if and only if all canonical horospheres are on one side from the common tangent plane to the horospheres  $L_i$, $L_j$ and $L_h$.
\end{prop}

By $B_i$, $B_j$ and $B_h$ denote the tangent points of $N$ with $L_i$, $L_j$ and $L_h$ respectively. We see that the prism $A_iA_jA_hB_hB_jB_i$ is a semi-ideal prism with lateral edges $r_i$, $r_j$ and $r_h$. It follows that the pair $(T, {\bf r})$ is admissible. We construct the complex $K=K(T, {\bf r})$.

It remains to check the convexity. Take two adjacent triangles $\Delta'=A_iA_jA_h$, $\Delta'' = A_jA_hA_g$ of $\hat T$ and corresponding semi-ideal prisms. In the construction above, points $A_i$, $A_j$, $A_h$ and $A_g$ lie in the same plane and the lateral faces of the prisms are not glued. To glue them, we should bend these prisms around the edge $A_jA_h$. The question is in which direction do we bend. 

Clearly, $\overline l_g \in \mathcal L_{M(\Delta')}$ if and only if the plane $N'=N(\Delta')$ coincides with the plane $N''=N(\Delta'')$, which is equivalent to the condition $\phi_i + \phi_g = \pi$ (edge $A_jA_h$ is flat). From now on assume that $N'$ and $N''$ are distinct planes.

Let $Y$ be the intersection point of $A_iA_g$ and $A_jA_h$. Parametrize the geodesic line $A_iA_g$ by length and let $y$ be the coordinate of $Y$. By $\rho$ denote the distance function $\rho_K$ restricted to $A_iA_g$. It has a kink point at $y$ and we need to check that it is concave. Let $\rho_1(x)$ and $\rho_2(x)$ be the distance functions from $A_iA_g$ to the planes $N'$ and $N''$ respectively. Hence, $\rho_1$ coincides with $\rho$ over $(-\infty; y]$ and $\rho_2$ coincides over $[y; +\infty)$. We have not proved yet that $N'$ and $N''$ are ultraparallel to $\H^2$. However, both $A_i$ and $A_g$ are in the same halfspaces $N'_-$ and $N''_-$. Therefore, the whole line $A_iA_g$ belongs to these halfspaces and $\rho_1$, $\rho_2$ have the form~(\ref{pd}). By Proposition~\ref{claim1}, the function $\rho_1(x) - \rho_2(x)$ has constant sign over the segments $(-\infty, y)$ and $(y, +\infty)$. If it is positive for $x>y$, then $\rho$ is concave at $y$ and $K$ is strictly convex at the edge $A_jA_h$.

Consider $x$ approaching $+\infty$. Take a sphere centered at the corresponding point $X \in A_iA_g$ (i.e. the set of points of $\H^3$ equidistant to $X$) tangent to $N''$. This sphere tends to the horosphere $L_g$ at $A_g$ as $x$ approaches $+\infty$. This horosphere belongs to the interior of $N'_-$, hence for some sufficiently large $x$, the sphere at $X$ does not intersect $N'$. It implies that $\rho_2(x)<\rho_1(x)$ and $A_jA_h$ is strictly convex.


We proved that if $T$ is Epstein-Penner for ${\bf r}$, then $(T, {\bf r})$ is a convex pair. Assume that $T'$ is another face triangulation of $K({\bf r}).$  According to Corollary~\ref{diffl}, $T$ and $T'$ can differ only in flat edges. By Lemma \ref{splcn}, faces of $K$ are ideal polygons, hence $T$ and $T'$ can be connected by a sequence of flips of flat edges. Let $T_k$ be an Epstein--Penner triangulation for $\bf r$ and $T_{k+1}$ be obtained from $T_k$ by flipping an edge $A_jA_h$ to $A_iA_g$. We saw before that an edge $A_jA_h$ between triangles $A_iA_jA_h$ and $A_jA_hA_g$ is flat if and only if all $\bar l_i$, $\bar l_j$, $\bar l_h$, $\bar l_g$ are in the same face of $C$.
This means that then $T_{k+1}$ also is Epstein--Penner for $\bf r$. 
\end{proof}

\subsection{Equivalence to Theorem~\ref{Luo}}
\label{DC}

Here we show that Theorem~\ref{angles} is equivalent to Theorem~\ref{Luo}. We need two facts. The first one is due to Akiyoshi~\cite{Aki}:

\begin{thm}
\label{ak}
For each hyperbolic cusp metric $d$ on $S_{g,n}$ there are finitely many Epstein-Penner triangulations of $(S_{g,n}, d)$.
\end{thm}

We remark that for our purposes it is enough to establish a weaker and easier fact that the number of triangulations is locally finite. For the sake of completeness we sketch a proof here. For ${\bf r}$ in a compact domain $R \subset \R^n$ the distance function $\rho_K$ of $K=K({\bf r})$ is bounded from below by a constant depending on $R$. By Corollary~\ref{trapdist}, the length of an upper edge of $K$ is bounded from above. But the lengths of geodesics between cusps on $(S_{g,n},d)$ form a discrete set (see e.g.~\cite{Pen}, Lemma 4.1).

Let $d'$ be a polyhedral hyperbolic metric on $S_{g,n}$ and $T$ be its geodesic triangulation. Denote the set of marked points by $\mathcal B=\{B_1, ..., B_n\}$. Take a triangle $B_iB_jB_h$. Clearly, there is a unique up to isometry semi-ideal prism that have $B_iB_jB_h$ as its lower face. Glue all such prisms together and obtain a complex $K(d', T)$ with the lower boundary isometric to $(S_{g,n}, d')$. Gluing isometries are uniquely defined if we fix the horosphere at each upper vertex passing through the respective lower vertex and match them together: one can see that this is the only way of gluing to obtain a complete metric space. 

\begin{lm}
\label{Le}
The complex $K(d', T)$ is convex if and only if $T$ is a Delaunay triangulation of $(S_{g,n}, d')$. Besides, any two convex complexes with isometric lower boundaries are isometric.
\end{lm}


\begin{proof}
In~\cite{Lei}, Section 3, Leibon provides a geometric observation showing that the intersection angle between circumscribed circles of two adjacent triangles $B_iB_jB_h$ and $B_jB_hB_g$ is equal to the dihedral angle of the upper edge $A_jA_h$. Clearly, a triangulation is Delaunay if and only if all these intersection angles are at most $\pi$. This gives us the first claim. Besides, if a diagonal switch transforms a Delaunay triangulation to Delaunay, then it is done in an inscribed quadrilateral and the Leibon observation shows that it switches a flat edge in the upper boundary and, thereby, does not change the complex. The fact that two Delaunay triangulations of $(S_{g,n}, d')$ can be connected by a sequence of diagonal switches through Delaunay triangulations is proved in~\cite{Lu}, Proposition 16. This settles the second claim. 

However, we remark that another proof of the second claim follows from similar ideas as our proof of Lemma~\ref{inj}. Indeed, for a point $X\notin \mathcal B$ in the lower boundary of $K=K(d', T)$ let $\sigma_K(X)$ be the length of the segment from $X$ to the upper boundary of $K$ orthogonal to its lower boundary. This is the same as distance function $\rho_K$ in Section~\ref{secinj}, but now defined on the lower boundary. We observe that the restriction of $\sigma_K$ to a geodesic parametrized by length has the form $$\sigma(x)=\arctanh(b\cosh(x-a))$$ for $0<b<1$ and analogues of all propositions for $\rho(x)$ of Section~\ref{secinj} can be obtained for $\sigma(x)$. The only difference is that $\sigma(x)$ is defined over an open bounded segment of $\R$, hence we should keep track on the domains of definition. Nevertheless, it does not provide substantial new difficulties and we omit the details. After that we can see that an ultraparallel prism is defined by its lower boundary and the lengths of lateral edges for non-ideal upper vertices. Therefore, the proof of the second claim can be finished in the same way as the proof of Lemma~\ref{inj}. We note that together with other observations of Subsection~\ref{secinj} it provides an alternative proof of Proposition 16 in~\cite{Lu} and of the fact that each Delaunay triangulation refines the Delaunay decomposition (see~\cite{Lu}, Section 2.4).
\end{proof}

Thus, we denote by $K(d')$ the (unique) convex complex that has $(S_{g,n}, d')$ as its lower boundary.

\begin{lm}
\label{equiv}
Theorem~\ref{Luo} is equivalent to Theorem~\ref{angles}.
\end{lm}

\begin{proof}
It is enough to show that $d'$ is discretely conformally equivalent to $d''$ if and only if the upper boundaries of $K(d')$ and $K(d'')$ are isometric.

Assume that the upper boundaries of $K(d')$, $K(d'')$ are both isometric to $(S_{g,n}, d)$ for a cusp metric $d$. Let $\mathcal K$ be the set of convex complexes realizing $(S_{g,n}, d)$. Choose a decoration on $(S_{g,n}, d)$ and identify $\mathcal K$ with $\R^n$ using Lemma~\ref{inj} and Lemma~\ref{sur}. First, assume that $K(d'), K(d'') \in \mathcal K(T)$ for a triangulation $T$. By Lemma~\ref{Le}, $T$ is Delaunay for both $d'$ and $d''$. Take $e \in E(T)$ and denote its lengths in $d'$ and $d''$ by $a'$ and $a''$ respectively. By $r'_i$ and $r'_j$ denote the weights of its endpoints in $K(d')$, by $r''_i$ and $r''_j$ in $K(d'')$. Then from Corollary~\ref{trap} we see that $$\sinh\left(\frac{a'}{2}\right)=\sinh\left(\frac{a''}{2}\right)\exp\left(\frac{r''_i-r'_i}{2}+\frac{r''_j-r'_j}{2}\right).$$
Thus, $d'$ is discretely conformally equivalent to $d''$. 

Assume that $d'$ and $d''$ are in the different cells $\mathcal K(T')$ and $\mathcal K(T'')$. The decomposition $\mathcal K = \bigcup \mathcal K(T)$ is finite due to Theorem~\ref{ak} and the boundaries of cells $\mathcal K(T)$ are piecewise analytic as subsets of $\R^n$. Then $K(d')$ and $K(d'')$ can be connected by a path in $\mathcal K$ transversal to the boundaries of all cells and intersecting them $m$ times. All intersection points correspond to distinct convex complexes. Denote their lower boundaries metrics by $d_1, \ldots, d_m$. Define also $d_0=d'$, $d_{m+1}=d''$. A segment between $d_i$ and $d_{i+1}$ of the path belongs to $\mathcal K(T_i)$ for some triangulation $T_i$. By Lemma~\ref{Le}, $T_i$ is Delaunay for both $d_i$ and $d_{i+1}$. By the previous argument, they are discretely conformally equivalent. Then so are $d_0$ and $d_{m+1}$. 

In the opposite direction, assume that $d'$ and $d''$ are discretely conformally equivalent and have a common Delaunay triangulation $T$. Then there exists a function $u: \mathcal B \rightarrow \R$ such that for each edge $e$ of $T$ with endpoints $B_i$ and $B_j$ we have $$\sinh\left(\frac{\len_{d'}(e)}{2}\right)=\exp(u(B_i)+u(B_j))\sinh\left(\frac{\len_{d''}(e)}{2}\right).$$ Consider $K(d')$ and $K(d'')$, then $T$ is a face triangulation of both these complexes due to Lemma~\ref{Le}. Metric spaces $(S_{g,n}, d')$ and $(S_{g,n}, d'')$ come with a homeomorphism between them isotopic to identity on $S_{g,n}$ with respect to $\mathcal B$. This allows us to identify the upper boundary metrics of $K(d')$ and $K(d'')$ with elements of the Teichmuller space of hyperbolic cusp metrics on $S_{g,n}$. Choose an horosection at each vertex of the upper boundaries in both $K(d')$, $K(d'')$. Let $r'_i$ and $r''_i$ be the distances from the horosections at $A_i \in \mathcal A$ to $B_i$ in $K(d')$ and $K(d'')$ respectively. We can choose the horosections such that for every~$i$, $\frac{r''_i-r'_i}{2}=u(B_i)$. Then Corollary~\ref{trap} shows that for each $e\in E(T)$ its length in the upper boundary of $K(d')$ is the same as in the upper boundary of $K(d'')$ (with respect to the chosen horosections). Therefore, the upper boundary metrics of $K(d')$ and $K(d'')$ together with the chosen decorations have the same Penner coordinates, hence they are isometric.

The case, when $d'$ and $d''$ are discretely conformally equivalent and do not have a common Delaunay triangulation $T$, is inductively reduced to the last case.
\end{proof}

\section{The variational approach} \label{VA}

In this section we prove Theorem~\ref{angles}. From a function $\kappa'$ we construct a functional $S_{\kappa'}$ such that its critical points in $\R^n$ are precisely complexes with curvatures prescribed by $\kappa'$. Then we show that it is strictly concave and that the Gauss-Bonnet condition on $\kappa'$ implies that it attains the maximal value in $\R^n$.

\subsection{The discrete Hilbert--Einstein functional}
\label{VA1}

For ${\bf r} \in \R^n$ let $T$ be any face triangulation of the convex complex $K(\bf r)$. We introduce \emph{the discrete Hilbert--Einstein functional} over the space of convex complexes $\mathcal K$ identified with $\R^n$:
\begin{equation}
\label{dHE}
S({\bf r}) := -2\vol(K({\bf r})) + \sum\limits_{1\leq i \leq n} r_i \tilde\kappa_i + \sum\limits_{e \in E(T)} l_e\tilde\theta_e.
\end{equation}

The curvatures $\tilde \kappa_i$ and exterior angles $\tilde \theta_e$ are measured in $K({\bf r})$. The value $S({\bf r})$ does not depend on the choice of $T$, because two face triangulations of $K(\bf r)$ are different only in flat edges, for which $\tilde \theta_e=0$.

Consider a function $\kappa': \mathcal A \rightarrow \R$. We write $\kappa'_i$ instead of $\kappa'(A_i)$. Define \emph{the modified discrete Hilbert-Einstein functional}:
\begin{equation}
\label{mdHE}
S_{\kappa'}({\bf r}) := S({\bf r})-\sum\limits_{1\leq i \leq n} r_i \kappa'_i.
\end{equation}

\begin{lm}\label{ppd}
For every ${\bf r} \in \R^n$, $S({\bf r})$ is twice continuously differentiable and
\begin{equation}\label{1pd}
\frac{\partial S}{\partial r_i} = \tilde\kappa_i.
\end{equation}
\end{lm}

\begin{proof}
Assume that ${\bf r}$ is an interior point of $\mathcal K(T)$ for some triangulation $T$. Then $T$ is the face decomposition of the upper boundary of $K({\bf r'})$  for every ${\bf r}'$ sufficiently close to ${\bf r}$. Hence, combinatorics of complexes does not change in some neighborhood of ${\bf r}$ and every total dihedral angle can be written as the sum of dihedral angles in the same prisms. Clearly, a dihedral angle in a non-degenerated prism is a smooth function of its edges. Moreover, by generalized Schl\"affli's differential formula (see \cite{Ri3}, Theorem 14.5) for a prism $P=A_iA_jA_hB_hB_jB_i \subset K$ we have
$$-2d {\rm vol} (P) = r_i d \omega_i + r_j d \omega_j + r_h d \omega_h + l_{jh} d \phi_i + l_{ih} d \phi _j + l_{ij} d \phi_h.$$

Summing these equalities over all prisms we obtain 
$$-2d\vol(K({\bf r}))=  - \sum\limits_{1\leq i \leq n} r_i d \tilde\kappa_i - \sum\limits_{e \in E(T)} l_e d\tilde\theta_e.$$

Thus, $$dS({\bf r})= \sum\limits_{1\leq i \leq n}  \tilde\kappa_i d r_i + \sum\limits_{e \in E(T)} \tilde\theta_e d l_e =\sum\limits_{1\leq i \leq n}  \tilde\kappa_i d r_i.$$

This gives (\ref{1pd}). Since dihedral angles in a prism are smooth functions of edges, we obtain that $S$ is twice continuously differentiable at ${\bf r}$.

Now consider the case when ${\bf r}$ belongs to the boundary of some $\mathcal K(T)$.
Let ${\bf e_i}$ be a coordinate vector. As the boundary of $\mathcal K(T)$ is piecewise analytic, ${\bf r} +\lambda{\bf e_i}\in\mathcal K(T)$ for some $T$ and  all small enough $\lambda$. Therefore, we can compute the directional derivative of $S({\bf r})$ in the direction ${\bf e_i}$ using the formula (\ref{1pd}). For every coordinate direction they are continuous, hence $S$ is continuously differentiable. In Lemma~\ref{concave} we show that the derivatives of $\kappa_i$ are also continuous, which will finish the proof that $S$ is twice continuously differentiable.
\end{proof}

\begin{crl} \label{ppdmf}
For every ${\bf r} \in \R^n$, $S_{\kappa'}({\bf r})$ is twice continuously differentiable and
$$\frac{\partial S_{\kappa'}}{\partial r_i} = \tilde\kappa_i-\kappa'_i.$$
\end{crl}

Corollary \ref{ppdmf} implies that if ${\bf r}$ is a critical point of $S_{\kappa'}$, then for all $i$, $\tilde \kappa_i=\kappa'_i$. In order to find it we investigate the second partial derivatives of $S_{\kappa'}$. It is sufficient to calculate them for a fixed triangulation $T$.

\begin{lm}
\label{concave}
Define $X_{ij}:=\frac{\partial^2 S}{\partial r_i\partial r_j} = \frac{\partial \kappa_i}{\partial r_j}$. Then for every $1 \leq i \leq n$:

(i) $X_{ii}<0,$

(ii) for $i \neq j$, $X_{ij} > 0,$

(iii) for every $1 \leq i \leq n$, $\sum\limits_{1\leq j \leq n} X_{ij} < 0$,

(iv) the second derivatives are continuous at every point ${\bf r} \in \R^n$. In particular, this implies that $X_{ij}=X_{ji}$.
\end{lm}

Note that a matrix satisfying the properties (i)--(iii) is a particular case of so-called \emph{diagonally dominated} matrices.

\begin{proof}
Let $A_1A_2A_3B_3B_2B_1$ be a semi-ideal prism. The solid angle at the vertex $A$ cuts a Euclidean triangle out of the canonical horosphere at $A$ with side lengths equal to $\alpha_{12}$, $\alpha_{13}$ and $\lambda$; its respective angles are $\phi_{12}$, $\phi_{13}$ and $\omega_1$. Then by the cosine law we have $$\cos(\omega_1) = \frac{\alpha^2_{12}+\alpha^2_{13} - \lambda^2}{2\alpha_{12}\alpha_{13}}.$$ We calculate the derivatives of $\omega_1$:

$$\frac{\partial \omega_1}{\partial \alpha_{12}} = -\frac{\cot(\phi_{12})}{\alpha_{12}},~~~~~~\frac{\partial \omega_1}{\partial \alpha_{13}} = -\frac{\cot(\phi_{13})}{\alpha_{13}}.$$

Calculate the derivatives of $\alpha_{12}$ from Corollary~\ref{alpha}:

$$\frac{\partial \alpha_{12}}{\partial r_1} = \frac{-\alpha^2_{12} - e^{-2r_1}}{2\alpha_{12}},~~~~~~\frac{\partial \alpha_{12}}{\partial r_2} = \frac{\alpha^2_{12} - e^{-2r_1}}{2\alpha_{12}}.$$

Consider a deformation of this prism fixing the upper face. Then

\begin{equation}
\label{ii}
\frac{\partial \omega_1}{\partial r_1} = \frac{\partial \omega_1}{\partial \alpha_{12} } \frac{\partial \alpha_{12}}{\partial r_1} + \frac{\partial \omega_1}{\partial \alpha_{13} }  \frac{\partial \alpha_{13}}{\partial r_1}=
\end{equation}
$$= \frac{\cot(\phi_{12})}{2\alpha^2_{12}}(\alpha^2_{12} + e^{-2r_1}) + \frac{\cot(\phi_{13})}{2\alpha^2_{13}}(\alpha^2_{13} + e^{-2r_1}),$$

\begin{equation}
\label{ij}
\frac{\partial \omega_1}{\partial r_2} = \frac{\partial \omega_1}{\partial \alpha_{12} } \frac{\partial \alpha_{12}}{\partial r_2} = \frac{\cot(\phi_{12})}{2\alpha^2_{12}}(-\alpha^2_{12} + e^{-2r_1}),\\
\end{equation}
$$
\frac{\partial \omega_1}{\partial r_3} = \frac{\partial \omega_1}{\partial \alpha_{13} } \frac{\partial \alpha_{13}}{\partial r_3} = \frac{\cot(\phi_{13})}{2\alpha^2_{13}}(-\alpha^2_{13} + e^{-2r_1}).
$$

Now consider a complex $K = K(T, {\bf r})$. Let $E^{or}(T)$ be the set of oriented edges of $T$: every edge $e \in E(T)$ gives rise to two oriented edges in $E^{or}(T)$. By $E^{orp}_i(T) \subset E^{or}(T)$ denote the set of oriented edges starting at $A_i$, but ending not in $A_i$. By $E^{orl}_i(T) \subset E^{or}(T)$ denote the set of oriented loops from $A_i$ to $A_i$ (thereby, every non-oriented loop is counted twice). By $E^{or}_i$ denote the union $E^{orp}_i(T) \cup E^{orl}_i(T).$ For an oriented edge $\vec e \in E^{or}_i(T)$ denote by $\alpha_{\vec e}$ the length of the arc of horosphere at $A_i$ between $A_iB_i$ and $\vec e$. To calculate $\frac{\partial \tilde{\omega}_i}{\partial r_i}$ we consider $\tilde{\omega}_i$ as the sum of angles in all prisms incident to $A_i$ and take their derivatives. If there are no loops among the upper edges of a prism, then this prism makes a contribution of the form~(\ref{ii}). If there are some loops, the derivative is obtained by summing a contribution of the form~(\ref{ii}) with contributions of the form~(\ref{ij}). Combining together the summands containing the terms $\alpha_{\vec e}$ we get



$$\frac{\partial \tilde \omega_i}{\partial r_i} = -X_{ii} =  \sum\limits_{{\vec e} \in E^{orp}_i(T)} \frac {\alpha^2_{\vec e} + e^{-2r_i}}{2\alpha^2_{\vec e}} (\cot\phi_{\vec e+} +\cot\phi_{\vec e-}) +$$ $$+ \sum\limits_{{\vec e} \in E^{orl}_i(T)} \frac{e^{-2r_i}}{\alpha^2_{\vec e}} (\cot\phi_{\vec e+} +\cot\phi_{\vec e-}),$$ where $\phi_{\vec e+}$ and $\phi_{\vec e-}$ are the dihedral angles at $\vec e$ in two prisms containing $\vec e$.

For every $e \in E(T)$ we have $$\phi_{\vec e+} +\phi_{\vec e-} = \tilde \phi_e \leq \pi,$$ where $e$ is $\vec e$ forgetting orientation. Hence $(\cot\phi_{\vec e+} +\cot\phi_{\vec e-}) \geq 0$ and $\frac{\partial \omega_i}{\partial r_i}=- X_{ii} \geq 0$. Also, equality here means that the total dihedral angle of every edge starting at $A_i$ is equal to $\pi$. But in this case we obtain a non-simply connected face of $K$, which is impossible by Lemma~\ref{splcn}. Similarly, for $i \neq j$ denote by $E^{orp}_{ij}(T) \subset E^{orp}_i$ the set of all oriented edges starting at $A_i$ and ending at $A_j$. Then using Corollary~\ref{alpha}

$$\frac{\partial \tilde\omega_i}{\partial r_j} = - X_{ij} =  \sum\limits_{{\vec e} \in E^{orp}_{ij}(T)} \frac { e^{-2r_i} - \alpha^2_{\vec e}}{2\alpha^2_{\vec e}} (\cot\phi_{\vec e+} +\cot\phi_{\vec e-}) =$$ $$= - \sum\limits_{{\vec e} \in E^{orp}_{ij}(T)} \frac { e^{r_j-r_i-l_e}}{2\alpha^2_{\vec e}} (\cot\phi_{\vec e+} +\cot\phi_{\vec e-}) < 0.$$

From this we obtain for every $i$,

$$\sum\limits_{1 \leq j \leq n}\frac{\partial \tilde\omega_i}{\partial r_j} =\sum\limits_{1 \leq j \leq n} - X_{ij} =  \sum\limits_{{\vec e} \in E^{or}_i(T)} \frac{e^{-2r_i}}{\alpha^2_{\vec e}}  (\cot\phi_{\vec e+} +\cot\phi_{\vec e-}),$$ which is greater than zero for similar reasons. It finishes the proof of Lemma~\ref{concave}.
\end{proof}

\begin{crl}
The functions $S$ and $S_{\kappa'}$ are strictly concave over $\R^n$.
\end{crl}

\begin{proof}
We show that the Hessian $X$ of $S$ is negatively definite over $\R^n$:
$${\bf r}^T X {\bf r} = \sum\limits_{i=1}^n X_{ii}r_i^2 + \sum\limits_{1\leq i < j \leq n} 2X_{ij} r_ir_j = -\sum\limits_{1\leq i < j \leq n}X_{ij}(r_i-r_j)^2 + \sum\limits_{i=1}^n r_i^2 \sum\limits_{j=1}^n X_{ij} <0.$$
\end{proof}

\begin{rmk}
One can see that the concativity of $S$ over a single cell $\mathcal K(T)$ follows from the fact that the volume of a prism is a concave function of its dihedral angles (proved in~\cite{Lei}) by applying of the Legendre transformation.
\end{rmk}

We see that $S_{\kappa'}({\bf r})$ has at most one maximum point. We want to prove that it exists. To this purpose we study what happens with complexes when the absolute values of some coordinates of ${\bf r}$ are large. First, we explore the case when all coordinates are sufficiently negative. Second, we deal with the case when there is at least one sufficiently large positive coordinate. Then we combine these results and get the desired conclusion.

\subsection{The behavior of $S_{\kappa'}$ near infinity}
\label{VA2}

\begin{lm}
\label{inf1}
For every $\varepsilon > 0$ there exists $C_1 > 0$ such that if for some $i$ we have $r_i < -C_1$ in ${\bf r} \in \R^n$, then $\tilde\omega_i < \varepsilon$  in $K({\bf r})$.
\end{lm}

\begin{proof}
Fix $\varepsilon > 0$. Recall that for $\vec e \in E^{or}_i(T)$ ending at $A_j$ (not necessarily different from $A_i$) Lemma~\ref{alpha} gives $$\alpha^2_{\vec e} = e^{r_j-r_i-l_{e}}+e^{-2r_i}.$$

Hence, $$\alpha^2_e \geq e^{-2r_i}.$$

Consider two consecutive edges $\vec e_1$ and $\vec e_2 \in E^{or}_i(T)$. Together with the line $A_iB_i$ they cut a Euclidean triangle out of the canonical horosphere at $A_i$ with the side length $\alpha_{\vec e_1}$, $\alpha_{\vec e_2}$ and $\lambda$. If $r_i < -C_1$, then both lengths $\alpha_{\vec e_1}$ and $\alpha_{\vec e_2}$ are at least $e^{C_1}$ and $\lambda$ is bounded from above by the total length of the canonical horocycle at $A_i$ on $(S_{g,n}, d)$. The angle $\omega$ between sides of lengths $\alpha_{\vec e_1}$ and $\alpha_{\vec e_2}$ in this triangle (which is the dihedral angle of $A_iB_i$ in the prism containing $\vec e_1$, $\vec e_2$ and $A_iB_i$) decreases as $C_1$ grows. We choose large enough $C_1>0$ such that if $r_i < -C_1$, then the angle at $A_iB_i$ in every such triangle is less than $\varepsilon/(6(n+2g-2))$.

Note that the number of triangles incident to one cusp is bounded from above by three times the total number of triangles of $T$, which can be calculated from the Euler characteristic and is equal to $2(n+2g-2)$. Therefore, the total dihedral angle $\tilde\omega_i < \varepsilon$.
\end{proof}

\begin{lm}
\label{D}
For every $\varepsilon > 0$ and $C_1>0$ there exists $C_2 > 0$ such that if for some $i$ we have $r_i \geq C_2$ and for every $j$ we have $r_j \geq -C_1$ in ${\bf r}\in \R^n$, then at every point $x \in S_{g, n}$ the value of distance function $\rho_K(x) \geq \varepsilon$, where $S_{g, n}$ is identified with the upper boundary of $K=K(\bf r)$.
\end{lm}

\begin{proof}
Let $J_i$ be the canonical horodisk at $A_i$ on $(S_{g,n}, d)$ and $G_i$ be its boundary. Below $sd(x, G)$ means the signed distance from a point to an horosphere on $(S_{g,n}, d)$. Our proof is based on the following simple propositions:

\begin{prop}\label{D1}
If $x, y \in S_{g,n}$, then ${\rho_K(x)\geq \rho_K(y)-d(x,y)}.$
\end{prop}

\begin{proof}
Let $\psi$ be a shortest geodesic arc connecting $x$ and $y$ in $(S_{g,n}, d)$. Define $N$ to be the union of segments in $K$ orthogonal to the lower boundary with one endpoint on $\psi$ and another one on the lower boundary. Then $N$ can be developed to $\H^3$. Let $M$ be the plane in $\H^3$ orthogonal to the lower endpoints of $N$. Clearly, $$\dist_{\H^3}(x, M) \geq \dist_{\H^3}(y, M)-\dist_{\H^3}(x,y)\geq \dist_{\H^3}(y, M)-d(x,y).$$ On the other hand, it is straightforward that $\dist_{\H^3}(x, M)=\rho_K(x)$ and $\dist_{\H^3}(y, N) = \rho_K(y)$.
\end{proof}

\begin{prop}\label{D2}
Let $t \in \R$ and $D(j, t) := \{x \in S_{g,n} \mid sd(x, G_j) \leq t\}$. If $x \in D(j, t)$, then $\rho_K(x) \geq r_j-t.$
\end{prop}

\begin{proof}
The set $D(j, t)$ is a horodisk centered at $A_j$. Let $T$ be a face triangulation of $K$. First, we consider the case when $t$ is small enough, so $D(j, t)$ is contained in the union of triangles of $T$ incident to $A_j$. If $x \in D(j, t)$, then in this case there is a triangle $A_iA_jA_h$ containing $x$. Develop the prism with this triangle to $\H^3$ and let $M$ be the plane passing through the lower boundary. The horodisk $D(j, t)$ is extended to the horoball $E$ in the development and $r_j-t$ is the signed distance from $E$ to $M$. We have $x \in E$, therefore $\rho_K(x)=\dist_{\H^3}(x, M)\geq r_j-t.$

If $D(j, t)$ does not meet this condition, then consider sufficiently small $t_0 \leq t$ such that $D(j, t_0)$ does. For each $x \in D(j, t)$ there exists $y \in D(j, t_0)$ such that $d(x,y)\leq t-t_0$ and $\rho_K(y)\geq r_j-t_0$. Then the desired bound follows from Proposition~\ref{D1}.
\end{proof}

Define $t:=-C_1-\e$ and $D:=\cup_{j \neq i} D(j, t)$, where $D(j,t)$ is defined in Proposition~\ref{D2}. Then this proposition implies that if $x \in D$, then ${\rho_K(x)\geq \e}$.

Define $p := \sup\{d(x, J_i): x \in S_{g,n}\backslash D\}.$ Note that $0 \leq p < \infty$. Indeed, if $x \in J_i$, then $d(x, J_i)=0$. But the closure of $S_{g,n}\backslash (D\cup J_i)$ is compact (possibly empty).

Take $C_2=\e+p\geq \e > 0$. If $x \in J_i$, then $\rho_K(x) \geq r_i \geq \e$ due to Proposition~\ref{D2}. If $x \notin D$, then there exists $y \in J_i$ such that $d(x,y) \leq p$ and $\rho_K(y)\geq r_i$. Then by Proposition~\ref{D1} we obtain $\rho_K(x) \geq \e$. This finishes the proof.
\end{proof}

From Corollary~\ref{trapdist} we see that

\begin{prop}
For every $\varepsilon > 0$ there exists $C > 0$ such that if the distance $\rho_{e}$ from an edge $e$ in the upper boundary of a complex $K$ to the lower boundary is greater than $C$, then $a_e < \varepsilon$, where $a_e$ is the length of the corresponding lower edge.
\end{prop}

Next proposition is straightforward:

\begin{prop}
For every $\varepsilon > 0$ there exists $\delta > 0$ such that if in a hyperbolic triangle every edge length is less than $\delta$, then the sum of its angles is greater than $\pi - \varepsilon$.
\end{prop}

Combining three last facts together we obtain

\begin{crl}
\label{inf2}
For every $\varepsilon > 0$ and $C_1>0$ there exists $C_2 > 0$ such that if for some $i$ we have $r_i \geq C_2$ and for every $j$ we have $r_j \geq -C_1$ in ${\bf r}\in \R^n$, then in $K({\bf r})$ $$\sum\limits_{1 \leq i \leq n} \tilde\omega_i \geq 2\pi(n+2g-2) -\varepsilon.$$
\end{crl}

Now we are able to prove that $S_{\kappa'}$ attains its maximal point at $\R^n$.

\begin{lm}
Consider a cube $Q$ in $\R^n$: $Q = \{ {\bf r} \in \R^n: \max(|r_i|) \leq q \}$. If 
\begin{equation}
\label{kappa}
\sum_{1\leq i\leq n}\kappa'_i > 2\pi(2-2g),
\end{equation}
then for sufficiently large $q$, the maximum of $S_{\kappa'}({\bf r})$ over $Q$ is attained in the interior of $Q$.
\end{lm}

\begin{proof}
Let $\mu_i:=2\pi-\kappa'_i$ and $\mu:=\min(\mu_i: 1\leq i \leq n)$. The condition~(\ref{kappa}) can be rewritten as 
$$2\pi(n+2g-2)>\sum_{1\leq i \leq n}\mu_i.$$
 
Take $C_1$ from Lemma \ref{inf1} for $\varepsilon = \mu$ and $C_2$ from Corollary \ref{inf2} for $C_1$ and $\e=\e_0$ where $$0< \varepsilon_0 < 2\pi(n+2g-2)-\sum_{1\leq i \leq n}\mu_i.$$ Let $q>\max\{C_1, C_2\}$. The cube $Q$ is convex and compact, $S$ is concave, therefore $S$ reaches its maximal value over $Q$ at some point ${\bf r}^0 \in Q$. Suppose that ${\bf r}^0 \in \textrm{bd}~Q$. Then there are two possibilities: either there is $i$ such that $r^0_i < -C_1 < 0$ or for every $i$ we have $r^0_i \geq -C_1$.

In the first case by Lemma \ref{inf1}, $\tilde\omega_i < \mu \leq \mu_i$. Therefore, $$\mu_i-\tilde\omega_i=\tilde\kappa_i - \kappa'_i = \left.\frac{\partial S_{\kappa'}}{\partial r_i} \right| _{{\bf r} = {\bf r}^0}> 0.$$ Let $\overline v_i$ be the $i$-th coordinate vector. We can see that for small enough $\nu>0$, $S_{\kappa'}({\bf r}^0 + \nu \overline v_i) > S_{\kappa'}({\bf r}^0)$ and ${\bf r}^0 + \nu \overline v_i \in Q$, which is a contradiction.

In the second case consider $i$ such that $|r^0_i| = q$. Then $r^0_i = q > C_2$ (because if $r^0_i=-q$, then the first case holds). Therefore, by Corollary \ref{inf2} we have $$\sum\limits_{1 \leq i \leq n} \tilde\omega_i \geq 2\pi(n+2g-2) -\varepsilon_0>\sum_{1\leq i\leq n} \mu_i.$$

Consider two sets $I = \{ i: 1\leq i \leq n, r^0_i = q \}$ and $J = [n] \backslash I$. Clearly, if $j \in J$, then $\left.\frac{\partial S_{\kappa'}}{\partial r_j} \right|_{{\bf r} = {\bf r}^0} = 0$. Therefore, $\tilde\omega_j = \mu_j$. Then we have $$\sum\limits_{i \in I} \tilde\omega_i >\sum_{i \in I}\mu_i.$$

Hence, for some $i \in I$ we obtain $\tilde\omega_i > \mu_i$ and so $\left.\frac{\partial S_{\kappa'}}{\partial r_i} \right|_{{\bf r} = {\bf r}^0}< 0$. Therefore, for small enough $\nu>0$, $S_{\kappa'}({\bf r}^0 - \nu \overline v_i) > S_{\kappa'}({\bf r}^0)$ and ${\bf r}^0 - \nu \overline v_i \in Q$, which is a contradiction.
\end{proof}

This finishes the proof of Theorem~\ref{angles}.

\bibliographystyle{abbrv}
\bibliography{fuchsian}
\end{document}